%% file: main.tex
\documentclass[10pt]{article}
\usepackage{amsmath}
\usepackage{amsthm}
\usepackage{amssymb}
\usepackage{amsfonts}
\usepackage{graphicx}
\usepackage{verbatim}
\usepackage{color}
\usepackage{caption,subcaption}
\usepackage{url}
\usepackage{ upgreek }
\usepackage[linktocpage,colorlinks=true]{hyperref}
\usepackage{thmtools,thm-restate}
\usepackage{todonotes}

\hypersetup{urlcolor=blue, citecolor=red, linkcolor=blue}

\theoremstyle{plain}
\newtheorem{theorem}{Theorem}[section]

\newtheorem{lemma}[theorem]{Lemma}
\newtheorem{corollary}[theorem]{Corollary}

\newtheorem{remark}[theorem]{Remark}

\theoremstyle{definition}
\newtheorem{definition}[theorem]{Definition}

\newcommand{\R}{\mathbb{R}}
\newcommand{\Z}{\mathbb{Z}}
\newcommand{\N}{\mathbb{N}}

\newcommand{\ba}{\bar{a}}
\newcommand{\bb}{\bar{b}}
\newcommand{\ha}{\hat{a}}
\newcommand{\hu}{\hat{u}}
\newcommand{\cL}{\mathcal{L}}
\newcommand{\cT}{\mathcal{T}}
\newcommand{\cLp}{\mathcal{L}^{(N)}}
\newcommand{\cLpk}{\mathcal{L}^{(N)}_k}
\newcommand{\cLpkinv}{(\cLp_k)^{-1}}

\newcommand{\cQ}{\mathcal{Q}}

\newcommand{\ta}{\tilde{a}}

\newcommand{\cI}{{\mathcal{I}}}
\newcommand{\bydef}{\stackrel{\mbox{\tiny\textnormal{\raisebox{0ex}[0ex][0ex]{def}}}}{=}}
\newcommand{\ball}[2]{\overline{B_{#1}\left( {#2} \right)}}

\newcommand{\eigk}{\mu_k}

\newcommand{\bL}{{\bf{\Lambda}}}

\newcommand{\bmuk}{{\boldsymbol{\mu_k}}}

\newcommand{\im}{{\bf{i}}}
\newcommand{\ro}[1]{row(#1)}
\newcommand{\co}[1]{col(#1)}
\newcommand{\ellmat}{1}

\title{Validated forward integration scheme for \\ parabolic PDEs via Chebyshev series}

\author{
Jacek Cyranka 
\thanks{University of Warsaw, Institute of Informatics, Stefana Banacha 2, 02-097 Warszawa, Poland. {\tt jcyranka@gmail.com}}
\and
Jean-Philippe Lessard\thanks{McGill University, Department of Mathematics and Statistics,  805 Sherbrooke Street West, Montreal, Qu\'ebec H3A 0B9, Canada. {\tt jp.lessard@mcgill.ca}}
}

\date{}

\begin{document}

\maketitle

\begin{abstract}
In this paper we introduce a new approach to compute rigorously solutions of Cauchy problems for a class of semi-linear parabolic partial differential equations. Expanding solutions with Chebyshev series in time and Fourier series in space, we introduce a zero finding problem $F(a)=0$ on a Banach algebra $X$ of Fourier-Chebyshev sequences, whose solution solves the Cauchy problem. The challenge lies in the fact that the linear part $\cL \bydef DF(0)$ has an infinite block diagonal structure with blocks becoming less and less diagonal dominant at infinity. We introduce analytic estimates to show that $\cL$ is an invertible linear operator on $X$, and we obtain explicit, rigorous and computable bounds for the operator norm $\| \cL^{-1}\|_{B(X)}$. These bounds are then used to verify the hypotheses of a Newton-Kantorovich type argument which shows that the (Newton-like) operator $\cT(a) \bydef a - \cL^{-1} F(a)$ is a contraction on a small ball centered at a numerical approximation of the Cauchy problem. The contraction mapping theorem yields a fixed point which corresponds to a classical (strong) solution of the Cauchy problem. The approach is simple to implement, numerically stable and is applicable to a class of PDE models, which include for instance Fisher's equation and the Swift-Hohenberg equation. We apply  our approach to each of these models. 
\end{abstract}



\section{Introduction} \label{sec:introduction}

In this paper, we introduce a new, fully spectral, validated forward integration scheme for a class of parabolic partial differential equations (PDEs) based on Chebyshev expansion in time. The class of PDE problems we consider is Cauchy problems associated with dissipative semi-linear equations of the form 
\begin{equation} \label{eq:general_PDE}
u_t = Lu + Q(u), \quad u(0,x) = u_0(x)
\end{equation}
where $u=u(t,x) \in \R$, $u_0(x)$ is a given initial condition, $x \in [0,2\pi]$, $t \ge 0$, $\partial_x^j = \frac{\partial^j}{\partial x^j}$, $L = \sum_{\ell=0}^d \gamma_{2\ell} \partial_x^{2\ell}$ ($\gamma_{2 \ell} \in \R$) is a linear differential operator of {\em even} order $2d$ and $Q(u) = \sum_{j=2}^p q_j u^j$ is a polynomial of degree $p$ in $u$ containing no constant term and no linear term. We supplement model \eqref{eq:general_PDE} with even boundary conditions, that is $u(t, -x) = u(t, x)$.

It is worth mentioning that the development of rigorous computational methods to study the flow of dissipative PDEs has received its fair share of attention in the last fifteen years. Let us mention the topological method based on covering relations and self-consistent bounds \cite{MR2049869,MR2788972,chaos_KS,MR3167726,MR3773757,Cy,CyZ}, the $C^1$ rigorous integrator of \cite{MR2728184}, the semi-group approach of \cite{MR3639578,MR3683781,TLJO,JLT_Schrodinger}, and the finite element discretization based approach of  \cite{nakao1,nakao2,nakao3}. This interest is perhaps not surprising as dissipative PDEs naturally lead to the notion of infinite-dimensional dynamical systems in the form of semi-flows, and understanding the asymptotic and bounded dynamics of these models is strongly facilitated by a rigorous investigation of the flow. While rigorous computations of periodic orbits may avoid the necessity of computing portions of the flow (they can indeed be obtained with Fourier expansions in time \cite{MR3662023,MR3623202,MR3633778,navier-stokes}), computing solutions to boundary values problems or connecting orbits often require a time integration. 

Our approach goes as follows. Expand the solution $u(t,x)$ as a Fourier series in $x$ with time-dependent Fourier coefficients. Obtain an infinite system of nonlinear ordinary differential equations (ODEs) to be solved on a time interval $[0,h]$. Using the Fourier coefficients of the initial condition $u_0(x)$, reformulate the ODEs as rescaled Picard integral equations over the time interval $[-1,1]$. Expand the solution of the integral equations with a Chebyshev series expansion in time. Derive an equivalent zero finding problem of the form $F(a)=0$ (where $a=(a_{k,j})_{k,j}$ is an infinite two-index sequence of Fourier-Chebyshev coefficients) whose solution correspond to the solution of the Cauchy problem \eqref{eq:general_PDE}. The operator $F$ is defined on a weighed $\ell^1$ Banach space $X$ of Fourier-Chebyshev coefficients
\[
X =
\{ 
a = (a_{k,j})_{k,j} ~ :  ~ \|a\|_X = \sum_{k,j} |a_{k,j}| \omega_{k,j} < \infty
\}.
\]
The weights $\omega_{k,j}$ in the definition of the norm $\| \cdot \|_{X}$ are chosen so that (a) they have geometric growth in $k$ (ensuring analyticity of the solutions in space, see Section~\ref{sec:regularity_C0_error}); and (b) $X$ is a Banach algebra under discrete convolutions. Next, let $\cL \bydef DF(0)$ be the Fr\'echet derivative of $F$ at $0 \in X$ and prove that $\cL$ is an invertible operator on $X$ (see Section~\ref{sec:linear_operator}). Then prove that the operator $\cT(a) \bydef a - \cL^{-1} F(a)$ is a contraction on a closed ball $\ball{r}{\ba}$ of radius $r>0$ centered at a numerical approximation $\ba \in X$. To obtain a proof that the operator $\cT:\ball{r}{\ba} \to \ball{r}{\ba}$ is a contraction for some explicit $r>0$, use a Newton-Kantorovich type theorem (Theorem~\ref{thm:radPolyBanach}) (combining functional analytic estimates and interval arithmetic computations) and the fact that the {\em step size} $h>0$ can be taken small if necessary. 
An application of the contraction mapping theorem yields a unique solution $\ta \in \ball{r}{\ba}$ of $F=0$, which represents the solution of the Cauchy problem on the time interval $[0,h]$.  The explicit radius $r>0$ yields a rigorous $C^0$ error bound between the true solution of the Cauchy problem \eqref{eq:general_PDE} and its numerical approximation (see Section~\ref{sec:regularity_C0_error}).

The main challenge of this approach is theoretical: show that the operator $\cL$ is invertible on the Banach space $X$ and obtain explicit and computable bounds for the operator norm $\| \cL^{-1} \|_{B(X)}$. As described in Section~\ref{sec:set-up}, the operator $\cL = (\cL_k)_k$ is a block diagonal operator, where each block $\cL_k$ acts on the sequence of Chebyshev coefficients of the Fourier mode $a_k(t)$, and consists of the sum of an infinite-dimensional tridiagonal operator and a rank one operator. To show that $\cL$ is an invertible operator, we show that each block $\cL_k$ is invertible on the $\ell^1$ Banach space of Chebyshev sequences. For a finite number of blocks $\cL_k$ with $k$ small, we use that $\cL_k$ is diagonal dominant starting from a moderately low Chebyshev dimension $N=N(k)$ to construct (with computer-assistance) an explicit approximate inverse $A_k$ for $\cL_k$ (see Figure~\ref{fig:operators_cL_kand_A_k}) which is then used in a Neumann series argument to get a rigorous bound on $\| \cL_k^{-1}\|_{B(\ell^1)}$ (see Lemma~\ref{lem:bounding_invLk_small_k}). 
As the Fourier dimension, $k$ grows, the Chebyshev projection number $N=N(k)$ from which the operator $\cL_k$ is diagonal dominant goes to infinity, and therefore the approach for small $k$ is not readily applicable. Hence, we derive an alternative and analytic approach to obtain a uniform bound $\| \cL_k^{-1}\|_{B(\ell^1)}$ for large $k$ (see Section~\ref{sec:linear_operator}), which  is based on the explicit inverse tri-diagonal operator analytic formulas introduced in \cite{cyranka_mucha}. Combining the computer-assisted technique for small $k$ and the one for large $k$, we obtain a rigorous bound for $\| \cL^{-1} \|_{B(X_{\nu,1})}$. We remark that forfeiting the diagonal dominance of $\cL_k$ for large $k$ has been the main obstacle in deriving a  forward integration scheme for parabolic PDE via the Chebyshev series. To the best of our knowledge, we present a first successful purely spectral approach for forward integration of parabolic PDEs via Chebyshev series (this is in contrast with the approach \cite{TLJO,JLT_Schrodinger,MR4356641} which also uses Chebyshev series expansions in time but handles the contraction mapping theorem via the semi-group flow action). Moreover, we believe that our technique of constructing explicit norm bounds of $\cL_k$ operator for large $k$ is of independent interest.

The novelty of our approach is threefold. First, it introduces a computational framework to handle infinite-dimensional problems with operators having the property that the two off-diagonal entries of $\cL_k$ are unbounded as $k$ grows (in contrast, the approach of \cite{tridiagonal} handles problems with tridiagonal operators having unbounded off-diagonal entries, but the operators are still diagonal dominant). Second, once the bound on $\| \cL^{-1} \|_{B(X_{\nu,1})}$ is obtained, the approach is rather straightforward to implement, computationally inexpensive and readily applicable to different models of the form \eqref{eq:general_PDE}. We stress that simple and efficient implementation is highly desired from the perspective of verifying code correctness of computer-assisted proofs in dynamics, namely clean and verifiable implementation is more likely to be widely accepted by the community. Third, contrary to time-stepping schemes, like the forward integration method based on the Taylor expansion and the Lohner algorithm from \cite{MR2788972,chaos_KS}, extended in \cite{MR3167726}, our approach is not burdened with the stiffness issue coming from the appearance of large positive and negative eigenvalues of the linear operator spectrum. The approximation quality of the Chebyshev series allows stable and high accuracy numerics (see the applications in Section~\ref{sec:applications}, where the considered examples have many - sometimes large - unstable eigenvalues). 

We must nevertheless confess that our approach has some limitations. The most important one is that the class \eqref{eq:general_PDE} is restrictive since it does not contain models having derivatives in the nonlinearity. We believe (based on numerical experimentation) that there are models (e.g. the Kuramoto-Sivashinsky equation and the phase-field crystal (PFC) equation, where the order of the derivative in the nonlinearity is small compared to the order $2d$ of the linear part $L$) for which our approach could be generalized and applied to. However, we do not foresee for the moment how to adapt our method to models like Burgers' equation, the Cahn-Hilliard equation, or the Ohta-Kawasaki model. There are two less worrisome limitations: (a) in the current setting, large step sizes $h$ are only possible when the norm of the solution is small (see Remark~\ref{rem:max_step_size}); and (b) there is a rapid propagation of wrapping effect from one step to the next, which prevents our approach to be applied iteratively for a large number of steps (see Remark~\ref{rem:wrapping_effect}). We believe that the step size restriction and the wrapping effect limitation can be overcome by extending our method even further (see Section~\ref{sec:future}). This is the subject of future investigation, and we are convinced that further research will lead to the successful elimination of the mentioned limitations. The goal of the present paper is to propose an alternative technique for forward integration of parabolic PDEs, which has not yet been explored and is based on novel principles. A simple implementation of the presented technique already demonstrated interesting experimental results and should lead to a new line of research in computer-assisted proofs in dynamics.

The paper is organized as follows. In Section~\ref{sec:set-up}, we derive the zero finding problem $F(a)=0$ whose solution corresponds to the solution of the Cauchy problem and we introduce a Newton-Kantorovich type argument to compute rigorously solutions to $F=0$. We demonstrate that the space-time regularity of the solution follows from the proof and that the solution so-obtained is classical (strong). In Section~\ref{sec:linear_operator}, we introduce a technique to show that $\cL$ is invertible on $X$ and we obtain explicit and computable bounds for the operator norm $\| \cL^{-1} \|_{B(X)}$. The computer-assisted method to obtain a rigorous bound on $\| \cL_k^{-1}\|_{B(\ell^1)}$ for small $k$ is presented in Section~\ref{sec:bounds_cL_k_inv_small_k}, while in Section~\ref{sec:uniform_bounds_large_k}, we introduce the analytic approach to obtain a uniform bound $\| \cL_k^{-1}\|_{B(\ell^1)}$ for large $k$. In Section~\ref{sec:radii_polynomial_bounds}, we present the construction of the necessary bounds to apply the Newton-Kantorovich type argument.  In Section~\ref{sec:applications}, we apply our approach to Fisher's equation and the Swift-Hohenberg equation. We conclude the paper by discussing future directions. 

\section{General set-up and a Newton-Kantorovich type argument} \label{sec:set-up}

This section begins in Section~\ref{sec:F=0}, where the derivation of the zero finding problem $F(a)=0$ is presented. Some properties of the Banach space on which we solve $F=0$ are introduced, and an equivalent fixed point formulation of the problem is presented. In Section~\ref{sec:radii_polynomial_approach} we present a Newton-Kantorovich type argument (see Theorem~\ref{thm:radPolyBanach}) which we use to solve $F=0$. We end in Section~\ref{sec:regularity_C0_error} by showing that the solution obtained from the fixed point argument is classical (strong), and by showing how to get a rigorous $C^0$ error bound between the exact solution and the numerical approximation of the Cauchy problem.

\subsection{The problem formulation and the Banach space} \label{sec:F=0}

Consider the general PDE \eqref{eq:general_PDE}, which we supplement with even boundary conditions (i.e. $u(t, -x) = u(t, x)$), in which case we expand solutions using a 
cosine Fourier series
\begin{equation} 
\label{eq:cosine_Fourier_expansion}
u(t,x) = \ta_0(t) + 2 \sum_{k \ge 1} \ta_k(t) \cos(k x) = \sum_{k \in \Z} \ta_k(t) e^{\im k x},  \quad \text{with } \ta_{-k}(t) = \ta_k(t) \in \R.
\end{equation}

After plugging the Fourier series \eqref{eq:cosine_Fourier_expansion} in \eqref{eq:general_PDE} the model reduces to the infinite system of ordinary differential equations
\begin{equation}\label{eq:ODEs_general}
\frac{d\ta_k}{dt} = f_k(\ta) \bydef \lambda_k \ta_k + Q_k(\ta), \quad \text{for all } k \ge 0,
\end{equation}
where the {\em eigenvalues} $\lambda_k \bydef \sum_{\ell=0}^d \gamma_{2\ell} (-1)^\ell  k^{2 \ell} \in \R$, and 
$
Q_k(\ta) \bydef \sum_{j=2}^p q_j (\ta^j)_k,
$
with $(\ta^j)_k = (\ta*\ta*\dots*\ta)_k$ denoting the discrete convolution of order $j$ (see \eqref{eq:discrete_convolution} for the definition of the discrete convolution).
Finally, by assumption, the semi-linear PDE model \eqref{eq:general_PDE} is dissipative, and therefore $\lim_{k \to \infty} \lambda_k = -\infty$.

In this paper, we propose to compute solutions of the Cauchy  problem associated to \eqref{eq:general_PDE} on a given time interval $[0,h]$, where $h>0$. This naturally leads to study the initial value problem
\begin{equation} \label{eq:IVP_general}
\frac{d}{dt} \ta_k(t) = f_k(\ta(t)), \quad \text{ for } t \in [0,h] \text{ and } \ta_k(0) = b_k \text{ for all } k\ge 0,
\end{equation}
where the vector $b = (b_k)_k$ corresponds to the Fourier coefficients of the initial condition $u_0(x)=u(0,x)$.
We rescale time by the factor $h>0$ to map the interval $[0,h]$ to $[-1,1]$ (letting $\tau \bydef 2t/h-1$ and $a_k(\tau) \bydef \tilde{a}_k(t) = \tilde{a}_k(\frac{h}{2}(\tau+1))$) so that 
\begin{align} \label{eq:rescale_ODEs}
\frac{d }{d \tau}a_k(\tau) &= \frac{h}{2} f_k(a(\tau)), \quad \text{ for } \tau \in [-1,1] \text{ and } a_k(-1) = b_k \text{ for all } k\ge 0.
\end{align}
Rewriting the system \eqref{eq:rescale_ODEs} as an integral equation results in
\begin{equation} \label{eq:integral_equation_general}
a_k(\tau) = b_k + \frac{h}{2} \int_{-1}^\tau f_k(a(s))~ds, \qquad k \geq 0, \quad \tau \in [-1,1].
\end{equation}
For each $k$, we expand $a_k(\tau)$ using a Chebyshev series, that is 
\begin{equation} \label{eq:a_k_Chebyshev_expansion}
a_k(\tau) = a_{k,0} + 2 \sum_{j \ge 1} a_{k,j} T_j(\tau) = a_{k,0} + 2 \sum_{j \ge 1} a_{k,j} \cos(j \theta) 
= \sum_{j \in \Z} a_{k,j} e^{\im j \theta}
= \sum_{j \in \Z} a_{k,j} T_j(\tau),
\end{equation}
where $a_{k,-j} = a_{k,j}$, $\tau = \cos(\theta)$ and $T_{-j}(\tau) \bydef T_j(\tau)$. The cosine Fourier expansion \eqref{eq:cosine_Fourier_expansion} becomes
\begin{equation} \label{eq:cosine_Fourier_Chebyshev_expansion}
u(\tau,x) =  \sum_{k,j \in \Z}  a_{k,j} e^{\im j \theta} e^{\im k x}, \quad a_{k,-j} = a_{k,j} \text{ and } a_{-k,j}=a_{k,j}.
\end{equation}
For each $k \ge 0$, we expand $f_k(a(\tau))$ using a Chebyshev series, that is 
\begin{equation} \label{eq:f_k(a)_Chebyshev_expansion}
f_k(a(\tau))
 = \phi_{k,0}(a)  + 2 \sum_{j \ge 1} \phi_{k,j}(a) \cos(j \theta) 
 =  \sum_{j \in \Z} \phi_{k,j}(a) e^{\im j \theta}=  \sum_{j \in \Z} \phi_{k,j}(a) T_j(\tau),
\end{equation}
where
\[
\phi_{k,j}(a) = \lambda_k a_{k,j} + Q_{k,j}(a).
\]
Letting $Q_k(a) \bydef (Q_{k,j}(a))_{j \ge 0}$, $\phi_k(a) \bydef (\phi_{k,j}(a))_{j \ge 0}$ and noting that $(\lambda_k a_k)_j = \lambda_k a_{k,j}$, we get that
\begin{equation} \label{eq:phi_expansion_general}
\phi_k(a) = \lambda_k a_k + Q_k(a).
\end{equation}

Combining \eqref{eq:integral_equation_general}, \eqref{eq:a_k_Chebyshev_expansion} and \eqref{eq:f_k(a)_Chebyshev_expansion} leads to %
\[
 \sum_{j \in \Z} a_{k,j} T_j(\tau) = a_k(\tau) = b_k + \frac{h}{2} \int_{-1}^\tau f_k(a(s))~ds =  b_k + \frac{h}{2} \int_{-1}^\tau \sum_{j \in \Z} \phi_{k,j}(a) T_j(s)~ds 
\]
and this results (e.g. see in \cite{MR3148084}) in solving $F=0$, where 
$F = \left(F_{k,j} \right)_{k,j}$ is given component-wise by
\[
F_{k,j}(a) = 
\begin{cases}
\displaystyle
 a_{k,0} + 2 \sum_{\ell=1}^\infty (-1)^\ell a_{k,\ell} - b_k, & j=0, k \ge 0 \\
\displaystyle
2j a_{k,j} + \frac{h}{2} ( \phi_{k,j+1}(a) - \phi_{k,j-1}(a)), & j>0,k \ge 0.
\end{cases}
\]
Hence, for $j>0$ and $k \ge 0$, we aim at solving
\[
F_{k,j}(a)  = 
2j a_{k,j} + \frac{h}{2} \lambda_k ( a_{k,j+1} - a_{k,j-1})
+ \frac{h}{2} ( Q_{k,j+1}(a) - Q_{k,j-1}(a)) = 0.
\]
Finally, the problem that we solve is $F=0$, where 
$F = \left(F_{k,j} \right)_{k,j}$ is given component-wise by
\begin{equation} \label{eq:F_{k,j}}
F_{k,j}(a) \bydef 
\begin{cases}
\displaystyle
 a_{k,0} + 2 \sum_{\ell=1}^\infty (-1)^\ell a_{k,\ell} - b_k, & j=0, k \ge 0 \\
\displaystyle
-\frac{h \lambda_k}{2} a_{k,j-1} + 2j a_{k,j} + \frac{h \lambda_k}{2} a_{k,j+1}
+ \frac{h}{2} (Q_{k,j+1}(a) - Q_{k,j-1}(a)), & j>0,k \ge 0.
\end{cases}
\end{equation}

Define the linear operator $\cL$ by 
\begin{equation} \label{eq:cL_{k,j}}
\cL_{k,j}(a) \bydef
\begin{cases}
\displaystyle
 a_{k,0} + 2 \sum_{\ell=1}^\infty (-1)^\ell a_{k,\ell}  , & j=0, ~ k \ge 0 \\
\displaystyle 
\mu_k a_{k,j-1} + 2 j a_{k,j} - \mu_k a_{k,j+1} , & j>0, ~ k \ge 0,
\end{cases}
\end{equation}
where 
\begin{equation} \label{eq:mu_k}
\mu_k \bydef - \frac{h}{2} \lambda_k.
\end{equation}
Note that $\lim_{k \to \infty} \mu_k = \infty$, and that $\mu_k>0$ except perhaps for a finite number of indices $k$. 
For a fixed Fourier component $k \ge 0$, the operator $\cL_k$ acts on $a_k \bydef (a_{k,j})_{j \ge 0}$ and 
can be visualized as 
\begin{equation} \label{eq:linear_operators}
\cL_{k} \bydef \begin{pmatrix}
1&-2&2&-2&2&\cdots \vspace{.1cm}
\\
\eigk&  2 & -\eigk&0&\cdots&\ \vspace{.1cm}
\\
0& \eigk& 4 & -\eigk&0&\cdots\\
\ &\ddots&\ddots&\ddots&\ddots&\ddots \vspace{.1cm} \\
\ &\dots&0& \eigk& 2j & -\eigk\\
\ &\ &\dots&\ddots\ddots& \ddots
\end{pmatrix}.
\end{equation}
Define the nonlinear operator $\cQ$ by 
\begin{equation} \label{eq:cQ_{k,j}}
\cQ_{k,j}(a) \bydef
\begin{cases}
\displaystyle
-b_k , & j=0, ~ k \ge 0 \\
\displaystyle  \frac{h}{2} \left(Q_{k,j+1}(a) - Q_{k,j-1}(a)  \right) , & j>0, ~ k \ge 0.
\end{cases}
\end{equation}
Setting
\begin{equation} \label{eq:tridiagonal_Lambda}
\Lambda \bydef
\begin{pmatrix} 
0&0&0&0&0&\cdots\\
-1&0&1&0&\cdots&\ \\
0&-1&0&1&0&\cdots\\
\ &\ddots&\ddots&\ddots&\ddots&\ddots\\
\ &\dots&0&-1&0&1 \\
\ &\ &\dots&\ddots&\ddots&\ddots
\end{pmatrix},
\end{equation}
we may write more densely the nonlinear part $\cQ_{k}$ as 
\[
\cQ_{k}(a) = -b_k + \frac{h}{2} \Lambda Q_k(a), \quad Q_k(a) \bydef (Q_{k,j}(a))_{j \ge 0}.
\]
Given a fixed Fourier mode $k \ge 0$, the formulation for $F$ in \eqref{eq:F_{k,j}} may be more densely written as
\begin{align*}
F_k(a) & = \cL_k a_k + \cQ_{k}(a) 
\\
& =  \cL_k a_k - b_k + \frac{h}{2} \Lambda Q_k(a),
\end{align*}
where $F_k(a) \bydef (F_{k,j}(a))_{j \ge 0}$. Finding $a$ such that $F_k(a)=0$ is equivalent (provided that the operator $\cL_k$ is invertible) to find a solution (fixed point) of 
\[
\cT_k(a) \bydef \cL_k^{-1} \left( b_{k} -  \frac{h}{2} \Lambda Q_k(a) \right) = a_k, \quad \text{for all } k \ge 0.
\]
Let us introduce the two block diagonal operators
\begin{equation} \label{eq:blockdiag}
\cL \bydef 
\begin{pmatrix}
\cL_{0}&0&\dots&0\\0&\ddots&0&\dots\\\ &\ddots&\cL_k&\ \\0&\dots&0&\ddots
\end{pmatrix}
\quad 
\text{and}
\quad
  \bL \bydef
  \begin{pmatrix}
    \Lambda&0&0&\dots\\
    0&\ddots&0&\dots\\
    \ &\ddots&\Lambda&\ \\
    \dots&0&0&\ddots
  \end{pmatrix}.
\end{equation}
Given $k \ge 0$, denote $a_k = (a_{k,j})_{j \ge 0}$. Denoting $a=(a_{0},\dots,a_k,\dots)$, we obtain
\[
\cL a = 
\begin{pmatrix} \cL_{0} a_{0} \\ \vdots \\ \cL_{k} a_k \\ \vdots \end{pmatrix}
\quad \text{and} \quad
\bL a = \begin{pmatrix} \Lambda a_{0} \\ \vdots \\ \Lambda a_k \\ \vdots \end{pmatrix}.
\]
We can finally write the map $F$ as 
\begin{equation} \label{eq:F(a)=0}
F(a) = \cL a - b + \frac{h}{2} \bL \cQ(a), 
\end{equation}
where it is understood that $b_k = (b_{k,j})_{j \ge 0}$ with $b_{k,j}=0$ for all $j>0$. 

The strategy we employ to prove existence of zeros of $F$, namely the Newton-Kantorovich type theorem presented in Section~\ref{sec:radii_polynomial_approach}, assumes that the map $F$ is Fr\'echet differentiable. This hypothesis is verified since the nonlinear term $Q$ in \eqref{eq:general_PDE} is assumed to be a polynomial in $u$.  

The Banach space $X=X_{\nu,1}$ in which we look for the zeros of $F$ is given by
\begin{equation} \label{eq:Banach_space}
X_{\nu,1} \bydef 
\left\{ 
a = (a_{k,j})_{k,j \ge 0} ~ :  ~ \|a\|_{X_{\nu,1}} \bydef \sum_{k,j \ge 0} |a_{k,j}| \omega_{k,j} < \infty
\right\},
\end{equation}
where $\nu \ge 1$ and
\begin{equation} \label{eq:norm_weights}
\omega_{k,j} \bydef 
\begin{cases}
1, & k=j=0 \\
2, & k=0, j>0 \\
2 \nu^{k}, & k > 0,j=0 \\
4 \nu^{k}, &k,j > 0.
\end{cases}
\end{equation}
The choice of the weights \eqref{eq:norm_weights} is to ensure that $X_{\nu,1}$ is a Banach algebra under discrete convolution, that is
\begin{equation} \label{eq:Banach_algebra}
\| a * b \|_{X_{\nu,1}} \le \| a \|_{X_{\nu,1}}  \| b \|_{X_{\nu,1}} 
\end{equation}
for all $a,b \in X_{\nu,1}$, where the discrete convolution of $a$ and $b$ is given by
\begin{equation} \label{eq:discrete_convolution}
(a * b)_{k,j} = \sum_{{k_1+k_2 = k \atop j_1+j_2=j} \atop k_i,j_i \in \Z} a_{k_1,|j_1|} b_{k_2,|j_2|},
\end{equation}
using the symmetries $a_{-k,j} = a_{k,j}$, $b_{-k,j} = b_{k,j}$ coming from the cosine Fourier expansion in space.

%
If the linear operator $\cL$ is invertible on $X_{\nu,1}$ (see Section~\ref{sec:linear_operator}), we may define the fixed point operator as 
\begin{equation} \label{eq:T(a)}
\cT(a) \bydef a - \cL^{-1} F(a) = \cL^{-1} \left( b - \frac{h}{2} \bL \cQ(a) \right).
\end{equation}

Denote by $\ell^1$ the Banach space 
\begin{equation} \label{eq:ell_one_norm}
\ell^1 \bydef \left\{ y= (y_j)_{j \ge 0} : \| y \|_{\ell^1} \bydef |y_0| + 2 \sum_{j \ge 1} |y_j|  < \infty \right\}.
\end{equation}
We can re-write
\[
\|a\|_{X_{\nu,1}} = \| a_0\|_{\ell^1} + 2 \sum_{k \ge 1} \|a_k \|_{\ell^1} \nu^k, \qquad a_k \bydef (a_{k,j})_{j \ge 0}.
\]
%


Having presented the problem formulations and the Banach space, we are ready to introduce a Newton-Kantorovich type theorem (sometimes called the radii polynomial approach) to prove the existence of fixed points of $\cT$ in $X_{\nu,1}$.

\subsection{A Newton-Kantorovich type theorem} \label{sec:radii_polynomial_approach}

Recall the map $F$ given by \eqref{eq:F(a)=0}, assume that $\cL:X_{\nu,1} \to X_{\nu,1}$ is invertible (see Section~\ref{sec:linear_operator}) and recall the fixed point operator $\cT$ given in \eqref{eq:T(a)}. 
Since $Q(u)$ in \eqref{eq:general_PDE} is polynomial, the map $F$ is Fr\'echet differentiable and therefore the map $\cT:X_{\nu,1} \to X_{\nu,1}$ is Fr\'echet differentiable. Denote by $D_a \cT(c)$ the Fr\'echet derivative of $\cT$ at a point $c\in X_{\nu,1}$.
Assume that a numerical approximation $\ba$ such that $\| F(\ba)\|_{X_{\nu,1}} \ll 1$ has been computed. Denote by
\[
\ball{r}{\ba} \bydef \left\{a \in X_{\nu,1}  ~:~  \| a - \ba \|_{X_{\nu,1}} \leq r \right\}
\]
the closed ball of radius $r>0$ centered in $\ba$ in $X_{\nu,1}$. Denote by $B(X_{\nu,1})$ the space of bounded linear operators on $X_{\nu,1}$ and $\| \cdot \|_{B(X_{\nu,1})}$ the induced operator norm.

\begin{theorem} \label{thm:radPolyBanach}
Let $Y$ and $Z=Z(r)$ be bounds satisfying
\begin{align}
\label{eq:Y_radPolyBanach}
\| \cT(\ba) - \ba \|_{X_{\nu,1}} &\le Y
\\
\label{eq:Z_radPolyBanach}
\sup_{c \in \ball{r}{\ba}} \| D_a \cT(c) \|_{B(X_{\nu,1})}  &\le Z(r).
\end{align}
Define the {\em radii polynomial}
\begin{equation} \label{eq:radPolyBanach}
p(r) \bydef r (Z(r)-1) + Y .
\end{equation}
If there exists $r_0>0$ such that
\[
p(r_0) < 0,
\]
then there exists a unique $\ha \in \ball{r_0}{\ba}$ satisfying $F(\ha) = 0$.
\end{theorem}

\begin{proof}
The idea is to show that $\cT$ is a contraction mapping of $\ball{r_0}{\ba}$ into itself, in which case the result follows from the contraction mapping theorem.

Let  $a \in \ball{r_0}{\ba}$ and apply the Mean Value Inequality to obtain  
\begin{align*}
\| \cT(a) - \ba \|_{X_{\nu,1}} & \leq \|\cT(a) - \cT(\ba) \|_{X_{\nu,1}}+ \| \cT(\ba) - \ba\|_{X_{\nu,1}} \\
&\leq \sup_{c \in \ball{r_0}{\ba}} \| D_a \cT(c)\|_{B(X_{\nu,1})} \| a - \ba\|_{X_{\nu,1}} + Y \\
&\leq r_0 Z(r_0) + Y,
\end{align*}
where the last inequality follows from \eqref{eq:Y_radPolyBanach} and \eqref{eq:Z_radPolyBanach}.
Using that $p(r_0) < 0$ implies that $\| \cT(a) - \ba \|_{X_{\nu,1}} < r_0$ and therefore that $\cT:\ball{r_0}{\ba} \to \ball{r_0}{\ba}$.

To see that $\cT$ is a contraction on $\ball{r_0}{\ba}$, let $c_1,c_2 \in \ball{r_0}{\ba}$ and see that
\begin{align*}
\| \cT(c_1) - \cT(c_2)\|_{X_{\nu,1}} & \leq \sup_{c \in \ball{r_0}{\ba}} \| D_a \cT(c)\|_{B(X_{\nu,1})} \| c_1 - c_2 \|_{X_{\nu,1}} \\
&\leq  Z(r_0) \| c_1 - c_2 \|_{X_{\nu,1}}.
\end{align*} 
Again, from the assumption that $p(r_0) < 0$ (that is $r_0 Z(r_0) + Y < r_0$), it follows that
\[
Z(r_0) < 1-  \frac{Y}{r_0} \le 1.
\]
Hence $\cT \colon \ball{r_0}{\ba} \to \ball{r_0}{\ba}$ is a 
contraction with contraction constant $Z(r_0)<1$. 
The contraction mapping theorem yields the existence of a unique $\ha \in \ball{r_0}{\ba}$ such that $\cT(\ha) = \ha - \cL^{-1} F(\ha)=\ha$. Since $\cL$ is invertible, $\cL^{-1}$ is invertible and this implies that $\ha$ is the unique element of $\ball{r_0}{\ba}$ satisfying $F(\ha)=0$.
\end{proof}

In Section~\ref{sec:radii_polynomial_bounds}, we construct explicitly the bounds necessary to apply the (radii polynomial) approach of Theorem~\ref{thm:radPolyBanach}. 

We conclude this section by introducing two consequences of a successful application of Theorem~\ref{thm:radPolyBanach}. The first one is that we get enough space and time regularity of the solution so that we obtain a classical solution to the Cauchy problem. The second one is that the radius $r_0>0$ such that $p(r_0)<0$ provides in fact a rigorous $C^0$ error control between the exact solution and a numerical approximation of the Cauchy problem.

\subsection{Regularity of the solutions and rigorous \boldmath$C^0$~\unboldmath error control} \label{sec:regularity_C0_error}

Assume that we applied the (radii polynomial) approach of Theorem~\ref{thm:radPolyBanach} to prove the existence of $\ha \in X_{\nu,1}$ such that $F(\ha)=0$ and $\| \ha - \ba \|_{X_{\nu,1}} \le r_0$ where $F$ is defined in \eqref{eq:F(a)=0}, $\nu \ge 1$ and $\ba$ is a numerical approximation. This is done by verifying that $r_0>0$ satisfies $p(r_0)<0$. Denote by
\begin{equation} \label{eq:tu}
\hu(\tau,x) \bydef
\sum_{k,j \in \Z}  \ha_{k,j} e^{\im j \theta} e^{\im k x}, \quad  \theta = \cos^{-1}(\tau), ~~ \ha_{k,-j} = \ha_{k,j} \text{ and } \ha_{-k,j}=\ha_{k,j}\end{equation}
the corresponding Fourier-Chebyshev expansion.

If $\nu>1$, then for each $k>0$,
$2 \|\ha_k \|_{\ell^1} \nu^k \le 
\| \ha_0\|_{\ell^1} + 2 \sum_{k \ge 1} \|\ha_k \|_{\ell^1} \nu^k = \| \ha \|_{X_{\nu,1}} < \infty$,
and therefore
\[
\|\ha_k \|_{\ell^1} \le \frac{\| \ha \|_{X_{\nu,1}}}{2 \nu^{k}}, \quad \text{for all } k>0,
\]
which has a geometric decay rate. Hence, $\hu$ is analytic in space and has enough spatial derivatives to be evaluated in the PDE model \eqref{eq:general_PDE}. If $\nu=1$, then by continuity of the bounds $Y$ and $Z(r)$ in the decay rate $\nu$, there exists $\epsilon>0$ such that $p(r)<0$ for some $\tilde \nu = 1 + \epsilon>1$, and therefore we are back to the previous case and space regularity follows (for a similar and more detailed argument, see Proposition 3 in \cite{MR3454370}). As for the time regularity, it follows from the fact that for each $k$, $\ha_k(\tau)$ is continuous in $\tau$ and solves the Picard integral equation \eqref{eq:integral_equation_general}. By continuity of $f_k$, $f_k(\ha(\tau))$ is continuous and therefore $\int_{-1}^\tau f_k(\ha(s))~ds$ is differentiable and therefore $\ha_k$ is differentiable in time. This follows that the resulting Fourier-Chebyshev expansion $\hu(\tau,x)$ given in \eqref{eq:tu} is a classical (strong) solution of \eqref{eq:general_PDE}.

Finally, denote by 
\[
\bar u(\tau,x) \bydef
\sum_{k,j \in \Z}  \ba_{k,j} e^{\im j \theta} e^{\im k x}, \quad \ba_{k,-j} = \ba_{k,j} \text{ and } \ba_{-k,j}=\ba_{k,j}
\]
the corresponding numerical approximate Fourier-Chebyshev expansion of the Cauchy problem. Then,
\begin{align} 
\nonumber
\| \hu - \bar u \|_{C^0} &\bydef
\sup_{\tau \in [-1,1] \atop x \in [0,2\pi]} |\hu(\tau,x) - \bar u(\tau,x)| \\
& \le \sum_{k,j \in \Z} |\ha_{k,j} - \ba_{k,j}| \left| e^{\im j \theta} e^{\im k x} \right| 
\nonumber
\\
\nonumber
& \le  \sum_{k \ge 0} \sum_{j \ge 0} |\ha_{k,j} - \ba_{k,j}| \omega_{k,j} \\
& = \| \ha - \ba \|_{X_{\nu,1}} \le r_0.
\label{eq:C0_error_bound}
\end{align}
This shows that the radius $r_0>0$ such that $p(r_0)<0$ (from the radii polynomial approach) provides in fact a rigorous $C^0$ error control between the exact solution $\hu$ and the numerical approximation $\bar u$ of the Cauchy problem.

We are now ready to introduce the theory to show that $\cL$ is invertible, to obtain rigorous estimates on $\| \cL_k^{-1}\|_{B(\ell^1)}$ for all $k \ge 0$ and finally to derive an explicit and computable bound for $\| \cL^{-1}\|_{B(X_{\nu,1})}$. 

\section{Analysis of the linear operator \boldmath$\cL$\unboldmath} \label{sec:linear_operator}

In this section, we introduce our approach to prove that the operator $\cL$ is invertible on $X_{\nu,1}$ and we obtain explicit and computable bounds for the operator norm $\| \cL^{-1} \|_{B(X_{\nu,1})}$.
In Section~\ref{sec:bounds_cL_k_inv_small_k}, we consider small $k$ and use that $\cL_k$ is diagonal dominant starting from a moderately low Chebyshev dimension $N=N(k)$ to construct (with computer-assistance) an explicit approximate inverse. The approximate inverse is used in a Neumann series argument to obtain a rigorous bound on $\| \cL_k^{-1}\|_{B(\ell^1)}$. Then in Section~\ref{sec:uniform_bounds_large_k}, we introduce an approach to obtain a uniform bound $\| \cL_k^{-1}\|_{B(\ell^1)}$ for large $k$. The approach here is also computer-assisted, utilizing both the numerical and symbolic computation. It is based on the explicit inverse tri-diagonal operator analytic formulas introduced in \cite{cyranka_mucha}. Combining the computer-assisted technique for small $k$ and the one for large $k$, we introduce in Section~\ref{sec:bound_cL_inv} a bound for $\| \cL^{-1} \|_{B(X_{\nu,1})}$.

\subsection{Bounds for \boldmath$\|\cL_k^{-1}\|_{B(\ell^1)}$~\unboldmath for small \boldmath$k$\unboldmath } \label{sec:bounds_cL_k_inv_small_k}

We begin this section by introducing some operators. Fix an even number $N$, and denote the operators $\tilde \Lambda$ and $\Omega$ acting on the tail of a Chebyshev sequence as 
\[
\tilde \Lambda \bydef
\begin{pmatrix} 
0&-1&0&\cdots&\ \\
1&0&-1&0&\cdots\\
0&1&0&-1&\cdots\\
&\ddots&\ddots&\ddots&\ddots&\ddots\\
&\dots&0&1&0&-1 \\
&\ &\dots&\ddots&\ddots&\ddots
\end{pmatrix}
\quad \text{and} \quad
\Omega \bydef
\begin{pmatrix} 
2(N+1)& 0 &0&\cdots  \\
0&2(N+2)&0&0 \\
0&0&2(N+3)&0 \\
\vdots &\vdots& &\ddots&\ddots  \\
\end{pmatrix}.
\]
Using the above {\em tail} operators, rewrite the operator $\cL_k$ in \eqref{eq:linear_operators} as in Figure~\ref{fig:operators_cL_kand_A_k}, where $\cL_k^{(N)} \in M_{N+1}(\R)$ is the matrix consisting of the first $(N+1) \times (N+1)$ entries of $\cL_k$.
Denote the infinite dimensional row vector $v \bydef \begin{pmatrix} -2&2&-2&2&-2&2& \cdots & \end{pmatrix}$, which is the tail of the first row of the operator $\cL_k$. 
Given a matrix $B$, denote by $B_{col(k)}$ the $k^{th}$ column of $B$. Similarly, $B_{row(k)}$ denotes the $k^{th}$ row of $B$. Define the operator $A_k$ (which acts as an approximate inverse for $\cL_k$) as given in Figure~\ref{fig:operators_cL_kand_A_k}.
\begin{figure}[h!]
\begin{center}
\includegraphics[width=15cm]{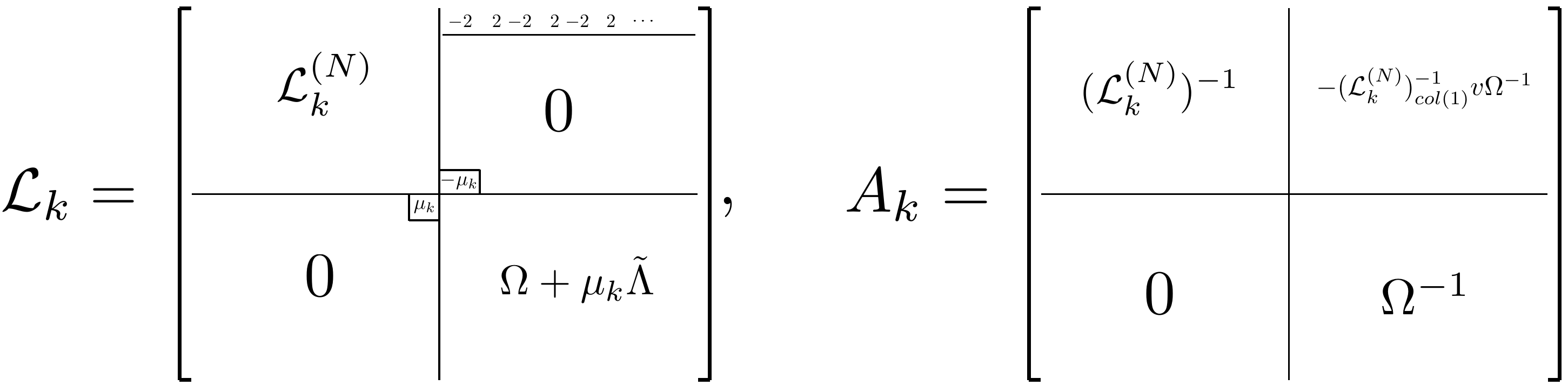}
\end{center}
\vspace{-.3cm}
\caption{The operator $\cL_k$ and its approximate inverse operator $A_k$.}
\label{fig:operators_cL_kand_A_k}
\end{figure}

Note that the choice of taking $N$ even comes from the fact that we want the first entry of the first row outside of $\cL_k^{(N)}$ to start with a $-2$. In other words, choosing $N$ even allows us to define the vector $v$ with a $-2$ as its first entry.

The following lemma provides an upper bound  for the operator norm in $B(\ell^1)$.

\begin{lemma} \label{lem:op_ell_one_norm}
Recall the definition of the $\ell^1$ norm in \eqref{eq:ell_one_norm}. Let $C=(c_{i,n})_{i,n \ge 0} \in B(\ell^1)$ and denote by $c_n$ the $n^{th}$ column of $C$, that is $c_n = (c_{i,n})_{i \ge 0}$. 
Then, 
\[
\| C\|_{B(\ell^1)} = \max \left\{\|c_n\|_{\ell^1}, \sup_{n \ge 1} \frac{1}{2} \|c_n\|_{\ell^1} \right\}.
\]
\end{lemma}

\begin{lemma} \label{lem:bounding_invLk_small_k}
Let $N \in \N$ be even and assume that $\cL_k^{(N)} \in M_{N+1}(\R)$ is invertible, and consider the operator $\cL_k$ as in Figure~\ref{fig:operators_cL_kand_A_k}. Let
\begin{align}
\rho^{(1)} &\bydef \frac{1}{N+1} \left( \|(\cL_k^{(N)})^{-1}_{col(1)}\|_{\ell^1} + 1 \right)
\label{eq:rho1}
\\
\rho^{(2)} &\bydef  \| (\cL_k^{(N)})^{-1}_{col(N+1)} + \frac{1}{N+2} (\cL_k^{(N)})^{-1}_{col(1)} \|_{\ell^1} + \frac{1}{N+2} 
\label{eq:rho2}
\\
\rho^{(3)} &\bydef \frac{2}{(N+1)(N+3)} \|(\cL_k^{(N)})^{-1}_{col(1)} \|_{\ell^1} + \frac{1}{N+1} + \frac{1}{N+3} 
\label{eq:rho3}
\end{align}
and
\begin{equation} \label{eq:rho_k}
\rho_k \bydef \frac{|\mu_k|}{2} \max \{ \rho^{(1)}, \rho^{(2)}, \rho^{(3)} \},
\end{equation}
where $(\cL_k^{(N)})^{-1}_{col(1)} \in \R^{N+1}$ denotes the first column of the matrix $(\cL_k^{(N)})^{-1}$. Let 
\begin{equation} \label{eq:beta_norm_A}
\beta_k \bydef \max \left\{ \| (\cL_k^{(N)})^{-1} \|_{B(\ell^1)} , \frac{1}{2(N+1)} \left( \| (\cL_k^{(N)})^{-1}_{col(1)} \|_{\ell^1} + 1 \right)  \right\}. 
\end{equation}
If $\rho_k<1$, then $\cL_k$ is a boundedly invertible operator on $\ell^1$ with
\begin{equation}
\| \cL_k^{-1} \|_{B(\ell^1)} \le \frac{\beta_k}{1-\rho_k}.
\end{equation}
\end{lemma}

\begin{proof}
From the definition of the operators $A_k$ and $\cL_k$ in Figure~\ref{fig:operators_cL_kand_A_k}, on can verify that the linear operator $I-A_k \cL_k$ is given in Figure~\ref{fig:operator_I_minus_AL}, where the operator $\widetilde{(\cL_k^{(N)})^{-1}_{col(N+1)}}$ is also defined in Figure~\ref{fig:operator_I_minus_AL}. 
%
%
\begin{figure}[h!]
\begin{center}
\includegraphics[width=15cm]{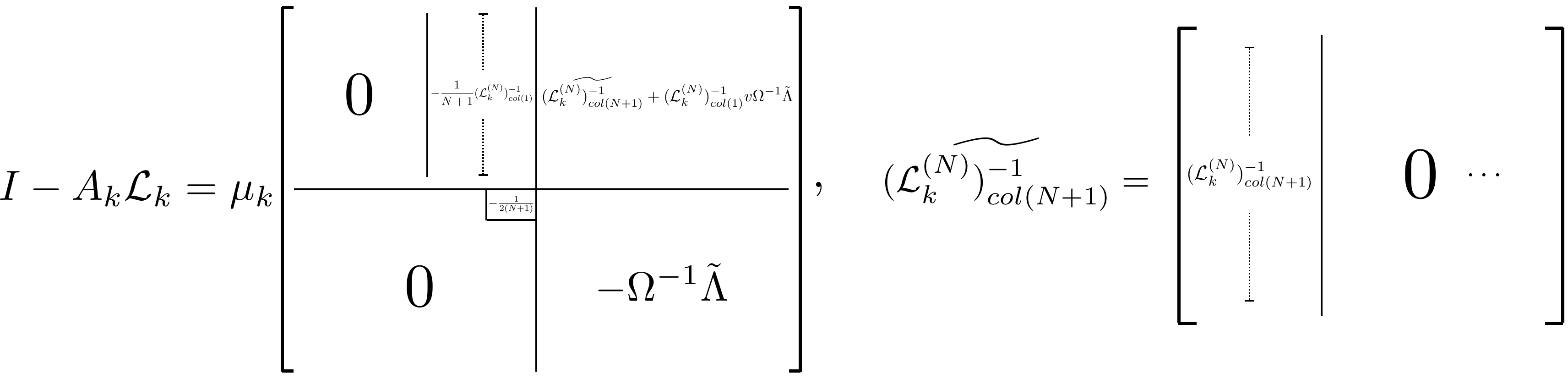}
\end{center}
\caption{The operators $I-A_k \cL_k$ and $\widetilde{(\cL_k^{(N)})^{-1}_{col(N+1)}}$.}
\label{fig:operator_I_minus_AL}
\end{figure}
%

Our first goal is to bound the operator norm $ \| I - A_k \cL_k \|_{B(\ell^1)}$. From Figure~\ref{fig:operator_I_minus_AL}, we see that the first $N$ columns of $I-A_k \cL_k$ are zero. Recalling Lemma~\ref{lem:op_ell_one_norm}, and denoting $C = I - A_k \cL_k$, we get that
\[
\| I - A_k \cL_k \|_{B(\ell^1)} = \| C \|_{B(\ell^1)} = \sup_{n \ge N+1} \frac{1}{2} \|c_n\|_{\ell^1}
\]

The $\ell^1$ norm of the $(N+1)^{th}$ column of $I-A_k \cL_k$ is given by $\|c_{N+1}\|_{\ell^1} = |\mu_k| \rho^{(1)}$, where $\rho^{(1)}$ is defined in \eqref{eq:rho1}. Let us now compute explicitly $c_j$, the $j^{th}$ columns of $I-A_k \cL_k$ for $j > N+1$. For this, we need to explicitly obtain the $j^{th}$ column of $(\cL_k^{(N)})^{-1}_{col(1)} v \Omega^{-1}\tilde \Lambda$. Let us first look at the {\em row} vector $v \Omega^{-1} \tilde \Lambda$, which is given by
\[
v \Omega^{-1} \tilde \Lambda = \left( \frac{1}{N+2},\frac{1}{N+1}-\frac{1}{N+3},-\frac{1}{N+2}+\frac{1}{N+4},\dots,\frac{2 (-1)^j}{(N+j-1)(N+1+j)},\dots \right),
\]
and therefore for $j>N+1$,
\begin{equation} \label{eq:Minv_col1_v_Omega_invT}
\left( (\cL_k^{(N)})^{-1}_{col(1)} v \Omega^{-1}\tilde \Lambda \right)_{col(j)} = 
\begin{cases}
\displaystyle
\frac{1}{N+2} (\cL_k^{(N)})^{-1}_{col(1)}, & j=N+2
\vspace{.2cm}
\\
\displaystyle
\frac{2 (-1)^j}{(j-2)j} (\cL_k^{(N)})^{-1}_{col(1)}, & j \ge N+3.
\end{cases}
\end{equation}
Hence, the finite part of the $(N+2)^{th}$ column of $\frac{1}{\mu_k}(I-A_k \cL_k)$ is given by the sum of (a) the first column of the operator $\widetilde{(\cL_k^{(N)})^{-1}_{col(N+1)}}$ (see Figure~\ref{fig:operator_I_minus_AL}), that is $(\cL_k^{(N)})^{-1}_{col(N+1)}$; and (b) the first column of the operator $(\cL_k^{(N)})^{-1}_{col(1)} v \Omega^{-1} \tilde \Lambda$, that is (using \eqref{eq:Minv_col1_v_Omega_invT}) $\frac{1}{N+2} (\cL_k^{(N)})^{-1}_{col(1)}$. Moreover, the {\em tail} part of the $(N+2)^{th}$ column of $\frac{1}{\mu_k}(I-A_k \cL_k)$ is given by $(0,-\frac{1}{2(N+2)},0,0,\dots)^T$. We conclude from this analysis that the $\ell^1$ norm of the $(N+2)^{th}$ column of $I-A_k \cL_k$ is given by $\|c_{N+2}\|_{\ell^1} = |\mu_k| \rho^{(2)}$, where $\rho^{(2)}$ is defined in \eqref{eq:rho2}. 

For $j>N+2$, using \eqref{eq:Minv_col1_v_Omega_invT}, the finite part of the $(N+2)^{th}$ column of $\frac{1}{\mu_k}(I-A_k \cL_k)$ is given by $\frac{2 (-1)^j}{(N+j-1)(N+1+j)} (\cL_k^{(N)})^{-1}_{col(1)}$ while the tail part is given by 
$(0,0,\dots,0,-\frac{1}{2(j-2)},0,\frac{1}{2j},0,0\dots)^T$. We conclude from this that for $j \ge N+3$, $\|c_{j}\|_{\ell^1}$, the $\ell^1$ norm of the $j^{th}$ column of $I-A_k \cL_k$ is bounded from above by $\|c_{N+3}\|_{\ell^1} = |\mu_k| \rho^{(3)}$, where $\rho^{(3)}$ is defined in \eqref{eq:rho3}. 

Therefore, we get that 
\[
\| I - A_k \cL_k \|_{B(\ell^1)} =  \rho_k = \frac{|\mu_k|}{2} \max \{ \rho^{(1)}, \rho^{(2)}, \rho^{(3)} \} < 1,
\]
by assumption. By a Neumann series argument, $A_k \cL_k$ is invertible with 
\[
(A_k \cL_k)^{-1} = \sum_{j \ge 0} ( I - A_k \cL_k)^j.
\]
By construction, $A_k$ is invertible and $\|A_k\|_{B(\ell^1)}=\beta_k$. Hence, $\cL_k$ is invertible with
\begin{equation} \label{eq:cL_k_inv_Neumann}
\cL_k^{-1} = \left( \sum_{j \ge 0} ( I - A_k \cL_k)^j \right) A_k.
\end{equation}
Using \eqref{eq:beta_norm_A}, we conclude the proof realizing that
\[
\| \cL_k^{-1} \|_{B(\ell^1)} \le  \sum_{j \ge 0} \left(\| I - A_k \cL_k\|_{B(\ell^1)}\right)^j  \| A_k\|_{B(\ell^1)} \le \frac{1}{1-\rho_k} \| A_k \|_{B(\ell^1)} =  \frac{\beta_k}{1-\rho_k}. \qedhere
\]
\end{proof}

We combined interval arithmetic with the result of Lemma~\ref{lem:bounding_invLk_small_k} 
 to obtain the following result, which when combined with the explicit tail bound of Corollary~\ref{cor:LkinvProj}, will be used to obtain a general bound for $\| \cL^{-1}\|_{B(X)}$.
 
 \begin{corollary} \label{cor:explicit_bound_norm_inv_small_mu}
 For all $\mu_k \in [0,1000]$, 
 \begin{equation}
 \| \cL_k^{-1}\|_{B(\ell^1)} \le 1.45.
 \end{equation}
 \end{corollary}
 
 \begin{proof}
Consider a mesh $0=\gamma_0<\gamma_1<\cdots< \gamma_{m-1}<\gamma_m = 1000$ of $[0,1000]$. Fix $j \in \cI \bydef \{0,\dots,m-1\}$, define the interval $\bmuk=\bmuk(j) = [\gamma_j,\gamma_{j+1}]$, and consider the $(N+1)\times(N+1)$ interval matrix $\boldsymbol{\cL_k^{(N)}}=\boldsymbol{\cL_k^{(N)}}(j)$ with $N=N(j)$, the finite part of $\cL_k$, whose tridiagonal entries are the intervals $\pm \bmuk$. With interval arithmetic, compute the interval valued matrix $(\boldsymbol{\cL_k^{(N)}})^{-1}$ and apply Lemma~\ref{lem:bounding_invLk_small_k}.
\end{proof}

\subsection{Uniform bound for \boldmath$\|\cL_k^{-1}\|_{B(\ell^1)}$~\unboldmath  for large \boldmath$k$\unboldmath} \label{sec:uniform_bounds_large_k}
\input{includes/linear-operators}
\subsection{Bound for \boldmath$\|\cL^{-1}\|_{B(X_{\nu,1})}$\unboldmath} \label{sec:bound_cL_inv}


Recall from Lemma~\ref{lem:bounding_invLk_small_k} the definitions of $\rho_k$ and $\beta_k$ given in \eqref{eq:rho_k} and \eqref{eq:beta_norm_A}, respectively. Assume that we have computed the bounds $\| \cL_k^{-1}\|_{B(\ell^1)} \le \frac{\beta_k}{1-\rho_k}$ for $k=0,\dots,\hat k$ using the computer-assisted approach of Section~\ref{sec:bounds_cL_k_inv_small_k}. Here, $\hat k$ is chosen so that $\mu_k \ge 0$ for all $k > \hat k$.

Combining Corollary~\ref{cor:explicit_bound_norm_inv_small_mu} and
Corollary~\ref{cor:Lkinv}, we obtain that 
\[
\| \cL_k^{-1} \|_{B(\ell^1)} \le 1.455, \quad \text{for all } \mu_k\ \ge 0. 
\]
Hence, denote by $\tilde \delta = 1.455$. 
Letting
\begin{equation} \label{eq:delta}
\delta \bydef \max \left(
\max_{k=0,\dots,\hat k} \frac{\beta_k}{1-\rho_k}
~,~
 \tilde \delta
 \right),
 \end{equation}
we get the following result.
\begin{lemma} \label{lem:final_bound_cL_inv}
\begin{equation} \label{eq:cL_inverse}
\| \cL^{-1} \|_{X_{\nu,1}} \le \delta.
\end{equation}
\end{lemma}

\begin{proof}
Letting $a \in X_{\nu,1}$ such that $\|a\|_{X_{\nu,1}} \le 1$, we get that
\begin{align*}
\| \cL^{-1} a \|_{X_{\nu,1}}  &= 
 \| \cL_{0}^{-1} a_0\|_{\ell^1} + 2 \sum_{k \ge 1} \| \cL_{k}^{-1} a_k \|_{\ell^1} \nu^k
\\
&\le 
 \| \cL_{0}^{-1} \|_{B(\ell^1)} \| a_0\|_{\ell^1} + 2 \sum_{k \ge 1}  \| \cL_{k}^{-1} \|_{B(\ell^1)} \| a_k \|_{\ell^1} \nu^k
\\
& \le \max \left(
\max_{k=0,\dots,\hat k} \| \cL_k^{-1} \|_{B(\ell^1)}
~,~
\tilde \delta
 \right) \left( \| a_0\|_{\ell^1} + 2 \sum_{k \ge 1} \|a_k \|_{\ell^1} \nu^k \right) \\
& \le  \delta \| h \|_{X_{\nu,1}} \le \delta. \qedhere
\end{align*}
\end{proof}

\section{Bounds for the radii polynomial} \label{sec:radii_polynomial_bounds}

Recall that the hypothesis of Theorem~\ref{thm:radPolyBanach} is verified using the radii polynomial $p(r)$ defined in \eqref{eq:radPolyBanach}. In this section, we present a constructive way to compute the bounds required to define $p(r)$, namely the bounds $Y$ and $Z=Z(r)$ given by \eqref{eq:Y_radPolyBanach} and \eqref{eq:Z_radPolyBanach}, respectively. 

Assume that the initial condition is given as
\[
b \in B_{r_0}(\bar b) = \left\{ b \in \ell_{\nu}^1 : \| b - \bb \|_{\ell_\nu^1} \le r_0 \right\},
\]
given some $\nu \ge 1$ and $r_0$ the error bound, and where the numerical initial condition  $\bb$ has only finitely many nonzero terms. Typically the error bound $r_0$ will come from the radius of the radii polynomial from the previous rigorous integration step. In Section~\ref{sec:initial_condition_for_next_step}, we show how this bound can be obtained. Here, it is understood that for $b = (b_k)_{k\ge 0}$, 
\[
\| b \|_{\ell_\nu^1} =  |b_0| + 2 \sum_{k \ge 1} |b_k| \nu^k.
\]
In this section, we use the notation $F(a,b)$ instead of $F(a)$ to emphasize the dependency of the zero finding problem $F=0$ (See \eqref{eq:F(a)=0}) in the initial condition $b \in \ell_\nu^1$.

Recall from Lemma~\ref{lem:bounding_invLk_small_k} the definitions of $\rho_k$ and $\beta_k$ given in \eqref{eq:rho_k} and \eqref{eq:beta_norm_A}, respectively. Assume that we have computed the bounds $\| \cL_k^{-1}\|_{B(\ell^1)} \le \frac{\beta_k}{1-\rho_k}$ for $k=0,\dots,\hat k$ using the computer-assisted approach of Section~\ref{sec:bounds_cL_k_inv_small_k}.
Denote by $\tilde \delta = 1.455$ the uniform bounds for $\| \cL_k^{-1}\|_{B(\ell^1)}$ for $k > \hat k$ from Section~\ref{sec:uniform_bounds_large_k}. 

We now present a strategy to compute each bound separately.

\subsection{The bound \boldmath$Y$\unboldmath} \label{sec:Y_general}

Recalling from \eqref{eq:T(a)} that $\cT(a) \bydef a - \cL^{-1} F(a,b) = \cL^{-1} \left( b - \frac{h}{2}  \bL \cQ(a) \right)$, note that $Y$ satisfies
\[
\| \cT(\ba) - \ba \|_{X_{\nu,1}} =
\| \cL^{-1} F(\ba,b)\|_{X_{\nu,1}} \le Y.
\]
Since the nonlinearity of the PDE is a polynomial and $\ba$ has only finitely many nonzero entries, there exists $M > \hat k$ such that $F_k(\ba,\bb)=0$ for all $k > M$.
Hence, 
\begin{align*}
\| \cL^{-1} F(\ba,\bb) \|_{X_{\nu,1}} &=  \|\cL_{0}^{-1} F_{0}(\ba,\bb)\|_{\ell^1} + 2 \sum_{k = 1}^{\hat k} \|\cL_{k}^{-1} F_{k}(\ba,\bb)\|_{\ell^1} \nu^k
+ 2 \sum_{k = \hat k+1}^{M} \|\cL_{k}^{-1} F_{k}(\ba,\bb)\|_{\ell^1} \nu^k
\\
& \le  \|\cL_{0}^{-1} F_{0}(\ba,\bb)\|_{\ell^1} + 2 \sum_{k = 1}^{\hat k} \|\cL_{k}^{-1} F_{k}(\ba,\bb)\|_{\ell^1} \nu^k
+ 2 \tilde \delta \sum_{k = \hat k+1}^{M} \| F_{k}(\ba,\bb)\|_{\ell^1} \nu^k.
\end{align*}
To compute a bound for $\| \cL_{k}^{-1} F_{k}(\ba,\bb)\|_{\ell^1}$ for $k=0,\dots,\hat k$, recall the definition of the approximate inverse $A_k$ in Figure~\ref{fig:operators_cL_kand_A_k} and recall \eqref{eq:cL_k_inv_Neumann} in the proof of Lemma~\ref{lem:bounding_invLk_small_k}, that is
\[
\cL_k^{-1} = \left( \sum_{j \ge 0} ( I - A_k \cL_k)^j \right) A_k.
\]
Hence,
we get that
\[
\| \cL_{k}^{-1} F_{k}(\ba,\bb)\|_{\ell^1}
= \left\|  \left( \sum_{j \ge 0} ( I - A_k \cL_k)^j \right) A_k F_{k}(\ba,\bb) \right\|_{\ell^1} \le \frac{1}{1-\rho_k} \| A_k F_{k}(\ba,\bb) \|_{\ell^1}.
\]
Hence, we can compute $Y_0$ such that
\begin{align*}
\| \cL^{-1} F(\ba,\bb) \|_{X_{\nu,1}} & \le \|\cL_{0}^{-1} F_{0}(\ba,\bb)\|_{\ell^1} + 2 \sum_{k = 1}^{\hat k} \|\cL_{k}^{-1} F_{k}(\ba,\bb)\|_{\ell^1} \nu^k
+ 2 \tilde \delta \sum_{k = \hat k+1}^{M} \| F_{k}(\ba,\bb)\|_{\ell^1} \nu^k \\
& \le \frac{\| A_0 F_{0}(\ba,\bb) \|_{\ell^1}}{1-\rho_0}  + 2 \sum_{k = 1}^{\hat k} \frac{\| A_k F_{k}(\ba,\bb) \|_{\ell^1} }{1-\rho_k} \nu^k
+ 2 \tilde \delta \sum_{k = \hat k+1}^{M} \| F_{k}(\ba,\bb)\|_{\ell^1} \nu^k \\
& \le Y_0.
\end{align*}

Moreover, recalling the definition of $\delta$ in \eqref{eq:delta}, note that
\[
\left\| \cL^{-1} (b-\bb)\right\|_{X_{\nu,1}}  \le \| \mathcal L^{-1} \|_{B(X_{\nu,1})} \| b-\bb \|_{X_{\nu,1}}
= \| \mathcal L^{-1} \|_{B(X_{\nu,1})}  \| b-\bb \|_{\ell_\nu^1} \le  \delta r_0.
\]

Finally, letting 
\begin{equation} \label{eq:Y_explicit}
Y \bydef Y_0 + \delta  r_0
\end{equation}
leads to the wanted bound as
\begin{align*}
\| \cT(\ba) - \ba \|_{X_{\nu,1}} 
& = \| \mathcal L^{-1} F(\ba,b) \|_{X_{\nu,1}} \\
& = \left\| \cL^{-1} \left( \cL \ba - b + \frac{h}{2}  \bL \cQ(\ba) \right) \right\|_{X_{\nu,1}}
\\
& \le \left\| \cL^{-1} \left( \cL \ba - \bb + \frac{h}{2}  \bL \cQ(\ba) \right) \right\|_{X_{\nu,1}} 
+ \left\| \cL^{-1} (b-\bb)\right\|_{X_{\nu,1}} 
\\
& \le  \| \cL^{-1} F(\ba,\bb) \|_{X_{\nu,1}} + \delta r_0 \\
& \le Y_0 + \delta r_0 = Y.
\end{align*}

\begin{remark}[\bf Wrapping effect] \label{rem:wrapping_effect}
In the current set-up, the bound \eqref{eq:Y_explicit} inevitably leads to a quick {\em wrapping effect}, as the error from the previous step is multiplied by the factor $\delta$, i.e. the upper bound \eqref{eq:delta} for $\|\cL^{-1}\|_{B(X_{\nu,1})}$. To exemplify this, assume that $\delta$ is constant along the integration and that after the first step, $r_0=\varepsilon$ (for some small $\varepsilon>0$). In this case, we expect the error bound $r_0 \ge \delta^k \varepsilon$ at step $k>1$.  For instance, if $\varepsilon = 10^{-14}$ and if $\delta = 1.5$, then at step $k>1$, the error bound should roughly be $10^{-14} 1.5^k$. Hence, expecting more than $k=80$ successful steps in this case if probably too ambitious, as in this case $10^{-14} 1.5^k>1$. See Tables~\ref{Tab:Fisher} and \ref{Tab:SH} for some explicit data. That being said, we believe that a multi-steps approach should significantly fix this problem. This approach is currently part of future research. 
\end{remark}

\subsection{The bound \boldmath$Z(r)$\unboldmath} \label{sec:Z_general}

Recall from \eqref{eq:Z_radPolyBanach} that the bound $Z(r)$ satisfies
\[
\sup_{c \in \ball{r}{\ba}} \| D_a \cT(c) \|_{B(X_{\nu,1})}  \le Z(r).
\]
The computation of the bound $Z(r)$ requires bounding the norm of some operators, which we do next.

\begin{lemma} \label{lem:bound_bL_norm}
Recall the definition of the operator $\Lambda$ in \eqref{eq:tridiagonal_Lambda} and the block diagonal operator $\bL$ in \eqref{eq:blockdiag}. Then $\bL \in B(X_{\nu,1})$ with
\begin{equation}
\| \bL \|_{B(X_{\nu,1})} \le 2.
\end{equation}
\end{lemma}

\begin{proof}
First, note that $\| \bL \|_{B(X_{\nu,1})} \le \| \Lambda \|_{B(\ell^1)}$, 
as for any $a \in X_{\nu,1}$ with $\| a \|_{X_{\nu,1}} \le 1$,  
\begin{align*}
\| \bL a \|_{X_{\nu,1}} &= \| \Lambda a_{0} \|_{\ell^1} + 2 \sum_{k\ge1} \| \Lambda a_{k} \|_{\ell^1} \nu^k
 \le    \| \Lambda \|_{B(\ell^1)} \| a_{0} \|_{\ell^1}
+ 2 \sum_{k\ge1} \| \Lambda \|_{B(\ell^1)} \|  a_{k} \|_{\ell^1} \nu^k \\
& \le \| \Lambda \|_{B(\ell^1)} \left(  \| a_{0} \|_{\ell^1} + 2 \sum_{k\ge1} \| a_{k} \|_{\ell^1} \nu^k  \right)  \le \| \Lambda \|_{B(\ell^1)} \| a \|_{X_{\nu,1}} 
\le \| \Lambda \|_{B(\ell^1)}.
\end{align*}
Now, let $b \in \ell^1$ such that $\| b \|_{\ell^1} \le 1$. Then, recalling the definition of the tridiagonal operator $\Lambda$ in \eqref{eq:tridiagonal_Lambda}, the proof follows by observing that
\begin{align*}
\| \Lambda b \|_{\ell^1}  &=
2 \sum_{k \ge 1} |-b_{k-1}+b_{k+1}| 
 \le 2 |b_0| + 2 \sum_{k \ge 2} |b_{k-1}| + 2  \sum_{k \ge 1} |b_{k+1}|  \\
& = \left( |b_0| + 2 \sum_{j \ge 1} |b_j| \right) + \left( |b_0| + 2 \sum_{j \ge 2} |b_j| \right)  \le 2 \|b \|_{\ell^1} \le 2. \qedhere
\end{align*}
\end{proof}

%
%

\begin{lemma} \label{lem:gamma_Z(r)}
Let $\gamma(r)$ be any finite bound satisfying
\begin{equation} \label{eq:gamma(r)}
\sup_{c \in \ball{r}{\ba}} \| D_a \cQ(c)  \|_{B(X_{\nu,1})} \le \gamma(r).
\end{equation}
Then
\begin{equation} \label{eq:Z(r)_general}
Z(r) \bydef h \delta \gamma(r)
\end{equation}
satisfies \eqref{eq:Z_radPolyBanach}.
\end{lemma}

\begin{proof}
The hypothesis that $Q=Q(u)$ is a polynomial implies that $\cQ(c)$ consists of discrete convolutions. 
Since $X_{\nu,1}$ is a Banach algebra under discrete convolutions, this implies that for any $r>0$,
\[
\sup_{c \in \ball{r}{\ba}} \| D_a \cQ(c)  \|_{B(X_{\nu,1})} < \infty.
\]
Now, letting $c \in \ball{r}{\ba}$ and using Lemma~\ref{lem:bound_bL_norm} and Lemma~\ref{lem:final_bound_cL_inv}, we get that
\begin{align*}
 \| D_a \cT(c) \|_{B(X_{\nu,1})} &=
 \| - \frac{h}{2} \mathcal \cL^{-1}  \bL D_a \cQ(c)  \|_{B(X_{\nu,1})}  \\
 & \le \frac{h}{2} \| \cL^{-1} \|_{B(X_{\nu,1})} \| \bL \|_{B(X_{\nu,1})} \| D_a \cQ(c)  \|_{B(X_{\nu,1})} \\
& \le h \delta  \| D_a \cQ(c)  \|_{B(X_{\nu,1})} \\
& \le h \delta \gamma(r) = Z(r). \qedhere
\end{align*}
\end{proof}


Note that the polynomial bound $\gamma(r)$ satisfying \eqref{lem:gamma_Z(r)} is problem dependent and will be computed explicitly for each of the PDE models we consider in Section~\ref{sec:applications}.

\begin{remark}[\bf A priori knowledge about a maximal step size] \label{rem:max_step_size}
Denoting
\begin{equation} \label{eq:Z1_general}
Z_1 \bydef Z(0) = h \delta \gamma(0) = h \delta
\| D_a \cQ(\ba)  \|_{B(X_{\nu,1})},
\end{equation}
a necessary condition for the radii polynomial approach to be successful is that $Z_1<1$. This is equivalent to require that 
\begin{equation} \label{eq:h_max}
h < h_{max} \bydef 
\frac{1}{\delta \| D_a \cQ(\ba)  \|_{B(X_{\nu,1})}}.
\end{equation}
From this observation, note that the larger $\|\ba\|_{B(X_{\nu,1})}$ is, the larger $\| D_a \cQ(\ba) \|_{B(X_{\nu,1})}$ is (indeed, $\cQ$ is a polynomial), and therefore smaller the step-size $h$ needs to be for the computer-assisted proof to be successful. 
\end{remark}

We explain in the section Applications (Section~\ref{sec:applications}) how to make use of the constraint \eqref{eq:h_max} to optimize our code. 

\subsection{Getting the bounds for the next initial condition} \label{sec:initial_condition_for_next_step}

Assume that at a previous time step, we computed $\ha \in \ball{r_0}{\ba} \bydef \left\{a \in X_{\nu,1}  \mid  \| a - \ba \|_{X_{\nu,1}} \leq r_0 \right\}$ such that $F(\ha) = 0$. The initial condition for the next step is then $b=(b_k)_{k \ge 0}$, where
\[
b_k \bydef \ha_k(1) = \ha_{k,0} + 2 \sum_{j \ge 1} \ha_{k,j}.
\]
Letting
\[
\bb_k \bydef \ba_{k,0} + 2 \sum_{j \ge 1} \ba_{k,j},
\]
then
\begin{align*}
\| b - \bb \|_{\ell_\nu^1} &=  |b_0 - \bb_0| + 2 \sum_{k \ge 1} |b_k - \bb_k| \nu^k \\
& = 
 \left| \ha_{0,0} + 2 \sum_{j \ge 1} \ha_{0,j} - \ba_{0,0} + 2 \sum_{j \ge 1} \ba_{0,j} \right| 
+ 2 \sum_{k \ge 1} \left| \ha_{k,0} + 2 \sum_{j \ge 1} \ha_{k,j} - \ba_{k,0} + 2 \sum_{j \ge 1} \ba_{k,j} \right| \nu^k \\
& \le \left( | \ha_{0,0} -  \ba_{0,0} | + 2 \sum_{j \ge 1} \left| \ha_{0,j} - \ba_{0,j} \right| \right) 
+ 2 \sum_{k \ge 1} \left| \ha_{k,0}  - \ba_{k,0}  \right| \nu^k
 + 4 \sum_{k,j \ge 1} \left| \ha_{k,j} - \ba_{k,j} \right| \nu^k \\
 & = 
 \sum_{k,j \ge 0} | \ha_{k,j} - \ba_{k,j} | \omega_{k,j}
  = \| \ha - \ba \|_{X_{\nu,1}} \le r_0.
\end{align*}

\section{Applications} \label{sec:applications}

In this section, we apply our approach to two models: Fisher's equation (in Section~\ref{sec:application_Fisher}) and the Swift-Hohenberg equation (in Section~\ref{sec:application_SH}).

To apply our approach (and Theorem~\ref{thm:radPolyBanach}) to each of the above PDE models, we compute the radii polynomial $p(r)$ defined in \eqref{eq:radPolyBanach}. For this we need the bounds $Y$ and $Z=Z(r)$ given by \eqref{eq:Y_radPolyBanach} and \eqref{eq:Z_radPolyBanach}, respectively. In Section~\ref{sec:Y_general}, we introduced the method to obtain the $Y$ bound in full generality in \eqref{eq:Y_explicit}. In Section~\ref{sec:Z_general}, recalling \eqref{eq:Z(r)_general}, we showed that $Z(r) \bydef h \delta \gamma(r)$ satisfied \eqref{eq:Z_radPolyBanach} with $\gamma(r)$ satisfying \eqref{eq:gamma(r)}. What remains to be done is to obtain explicitly the polynomial bound $\gamma(r)$ for each model, which we do next.

Before doing that, we briefly describe the optimization of the step-size we perform before each computer-assisted proof.

\subsection{Procedure for optimizing the step-size before a computer-assisted proof} \label{sec:optimizing_step_size}

Recalling Remark~\ref{rem:max_step_size} about the maximal step size, we now present a simple heuristic to optimize the step-size before attempting a computer-assisted proof. Recall the definition of the bound $Z_1$ in \eqref{eq:Z1_general}, and fix a target value $Z_1^{(\rm target)}<1$ that we want to achieve for $Z_1$ and a tolerance ${\tt tol}$. In all of the examples below, we chose $Z_1^{(\rm target)}=0.7$ and ${\tt tol} = 0.01$. Fix an initial tentative step-size $h_0>0$, and start the procedure.

Compute a solution $\ba$ of $F(\ba,b)=0$ using an iterative procedure (we use pseudo-Newton $a \mapsto a - \cL^{-1} F(a,b)$, which avoids having to compute numerically $DF(\ba)$ and $DF(\ba)^{-1}$). Make sure that the last Chebyshev coefficients of each Fourier modes $a_k(t)$ are of the order of machine precision ($\approx 10^{-16}$). Then compute (without interval arithmetic) the bound $\delta = \delta(h_0)$ given by formula \eqref{eq:delta}. Using \eqref{eq:Z1_general}, compute 
\[
Z_1 = h \delta \| D_a \cQ(\ba)  \|_{B(X_{\nu,1})},
\]
where the bound $\| D_a \cQ(\ba)  \|_{B(X_{\nu,1})}$ is easily obtained using Banach algebra estimates (e.g. see the explicit $Z_1$ bounds  \eqref{eq:Z1_Fisher} and \eqref{eq:Z1_SH} for each model we consider). 

\begin{itemize}
\item If $Z_1>Z_1^{(\rm target)}$ and $|Z_1-Z_1^{(\rm target)}| > {\tt tol}$, replace $h_0 \mapsto 0.9 h_0$ and start from the beginning.
\item If $Z_1<Z_1^{(\rm target)}$ and $|Z_1-Z_1^{(\rm target)}| > {\tt tol}$, replace $h_0 \mapsto 1.01 h_0$ and start from the beginning.
\item If $|Z_1-Z_1^{(\rm target)}| \le {\tt tol}$, then stop the procedure.
\end{itemize}

Repeat the steps until the wanted tolerance {\tt tol} is achieved or until you have reached an a priori fixed maximal number of steps. 

We are now ready to present some applications. 

\subsection{Fisher's equation} \label{sec:application_Fisher}

Fisher's equation is given by
\begin{equation} \label{eq:Fisher}
u_t = u_{xx} + \alpha u - \alpha u^2, \quad \alpha \in \R
\end{equation}
and it has applications in mathematical ecology, genetics, and the theory of Brownian motion \cite{MR1423804,MR639998,MR0400428}. We supplement Fisher's equation with even (i.e. $u(t, -x) = u(t, x)$) boundary conditions, plug \eqref{eq:cosine_Fourier_expansion} in \eqref{eq:Fisher} and this leads to the following infinite system of ordinary differential equations 
\begin{equation} \label{eq:ODEs_Fisher}
\frac{d \ta_k}{dt} = f_k^{({\rm F})}(\ta) \bydef (-k^2 + \alpha) \ta_k - \alpha (\ta*\ta)_k, \quad k \ge 0.
\end{equation}
Recalling \eqref{eq:ODEs_general}, we get for Fisher that $\lambda_k = -k^2 + \alpha$ and $Q_{k}(a) = - \alpha (a^2)_k$. 

\subsubsection{The bound \boldmath$\gamma(r)$~\unboldmath and \boldmath$h_{max}$\unboldmath}

Given any $c \in \ball{r}{\ba}$ and $h \in \ball{1}{0}$, note that $D_a \cQ(c) h = - 2 \alpha (c*h)$. Therefore 
\[
\| D_a \cQ(c)\|_{B(X_{\nu,1})} \le \sup_{h \in \ball{1}{0}} 2 |\alpha| \| c\|_{X_{\nu,1}} \| h \|_{X_{\nu,1}}
\le \gamma(r) \bydef
2 |\alpha| \left( \| \ba\|_{X_{\nu,1}} + r \right).
\]
Recalling Lemma~\ref{lem:gamma_Z(r)}, the polynomial $\gamma(r)$ satisfies \eqref{eq:gamma(r)} and \eqref{eq:Z(r)_general}, the bound $Z(r)$ for Fisher is given by
\begin{equation} \label{eq:Z(r)_Fisher}
Z(r) = 2h \delta |\alpha| \| \ba\|_{X_{\nu,1}} + 2h \delta |\alpha| r
\end{equation}
and 
\begin{equation} \label{eq:Z1_Fisher}
Z_1 = Z(0) = 2h \delta |\alpha| \| \ba\|_{X_{\nu,1}} .
\end{equation}
Hence, when applying the procedure for optimizing the step-size before a computer-assisted proof (see Section~\ref{sec:optimizing_step_size}), we have that
\[
h < h_{max} = \frac{1}{2 \delta |\alpha| \| \ba\|_{X_{\nu,1}}}.
\]

We fixed the parameter value in Fisher's equation to be $\alpha = 100$, at which there are $10$ unstable eigenvalues $\lambda_k \in \{19,36,51,64,75,84,91,96,99,100\}$. We fixed the initial condition to be $u_0(x) = -0.1 + 0.02 \cos(x) -0.002 \cos(2x)$. For the whole integration, we fixed the number of Fourier coefficients to be $20$. We report the results in Table~\ref{Tab:Fisher} and in Figure~\ref{fig:Fisher_100}.

\begin{table}[ht]
\centering
{
\begin{tabular}{ccccccc}
\hline
Steps & $h$ & $\#$ of Cheb. coeff. & $\delta$ & $r_0$  \\
\hline\\[-3mm]
$1$ & $4.5001 \times 10^{-3}$ & $17$ & $1.571$ & $1.6371 \times 10^{-13}$ \\
$2$ & $3.2806 \times 10^{-3}$ & $17$ & $1.455$ & $5.6452 \times 10^{-13}$ \\
$3$ & $2.3915 \times 10^{-3}$ & $16$ & $1.455$ & $1.4908 \times 10^{-12}$ \\
$4$ & $1.7963 \times 10^{-3}$ & $16$ & $1.455$ & $3.4225 \times 10^{-12}$ \\
$5$ & $1.4550 \times 10^{-3}$ & $16$ & $1.455$ & $7.5886 \times 10^{-12}$ \\
$6$ & $1.1785 \times 10^{-3}$ & $16$ & $1.455$ & $1.6416 \times 10^{-11}$ \\
$7$ & $9.5459 \times 10^{-4}$ & $15$ & $1.455$ & $3.4419 \times 10^{-11}$ \\
$8$ & $7.9665 \times 10^{-4}$ & $15$ & $1.455$ & $7.1317 \times 10^{-11}$ \\
$9$ & $7.1698 \times 10^{-4}$ & $15$ & $1.455$ & $1.5120 \times 10^{-10}$ \\
$10$ & $5.8076 \times 10^{-4}$ & $15$ & $1.455$ & $3.1155 \times 10^{-10}$ \\
$15$ & $2.9487 \times 10^{-4}$ & $15$ & $1.455$ & $1.2089 \times 10^{-8}$ \\
$20$ & $1.6470 \times 10^{-4}$ & $15$ & $1.455$ & $4.6332 \times 10^{-7}$ \\
$25$ & $9.7251 \times 10^{-5}$ & $15$ & $1.455$ & $1.7835 \times 10^{-5}$ \\
$30$ & $5.7426 \times 10^{-5}$ & $15$ & $1.455$ & $7.0234 \times 10^{-4}$ \\
$35$ & $3.4249 \times 10^{-5}$ & $15$ & $1.455$ & $2.6702 \times 10^{-2}$ \\

\hline 
\end{tabular}
}
\caption{Data for the rigorous enclosure of the solution of the Cauchy problem for Fisher's equation.}
\label{Tab:Fisher}
\end{table}

\begin{figure}[h!]
\begin{center}
\includegraphics[width=8cm]{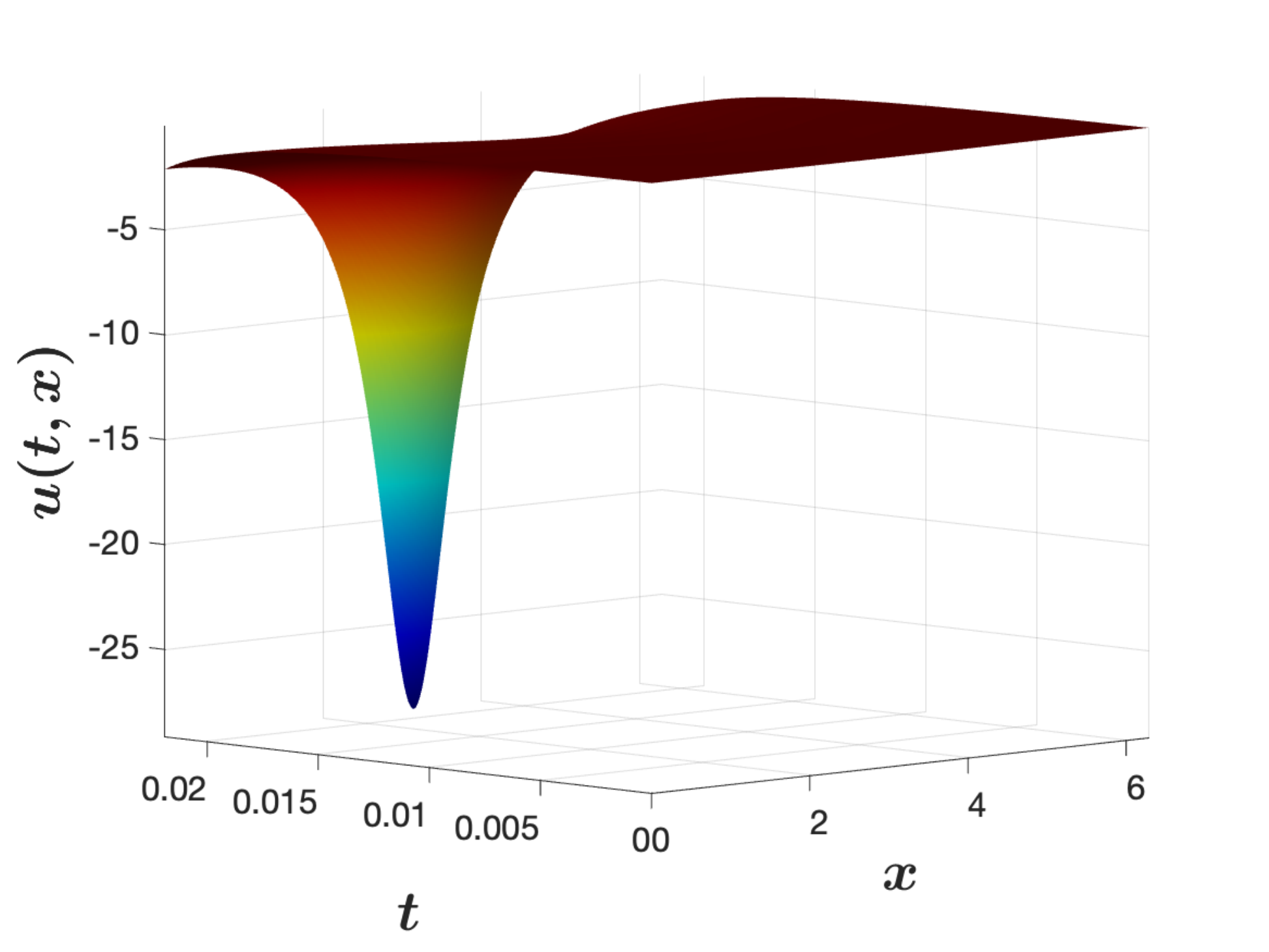}
\includegraphics[width=7cm]{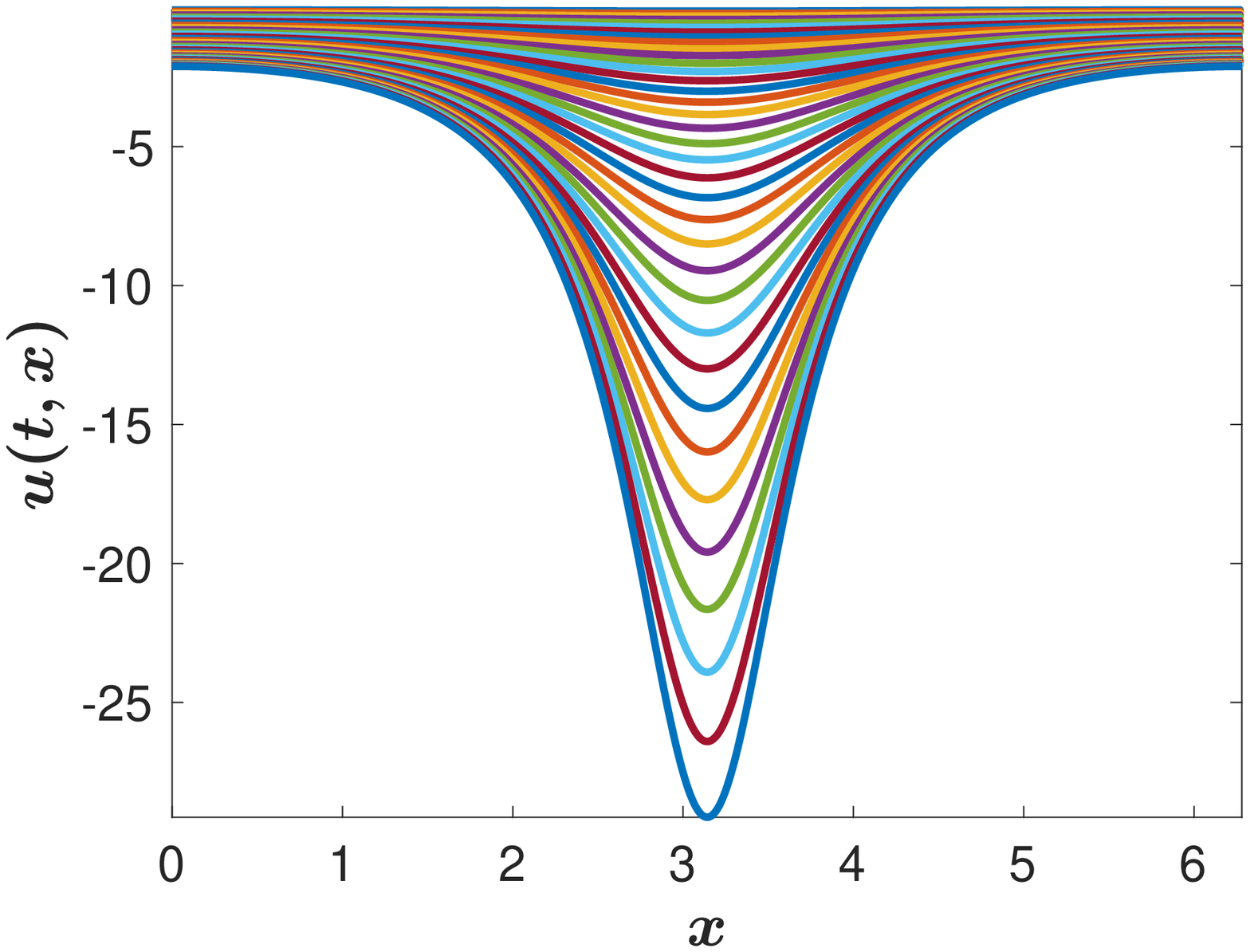}
\end{center}
\vspace{-.3cm}
\caption{The solution of the Cauchy problem for Fisher with $\alpha = 100$ and $u_0(x) = -0.1 + 0.02 \cos(x) -0.002 \cos(2x)$. 
The number of steps is $35$ and the total integration time is $0.021895$.
There are $10$ unstable eigenvalues which are given by $\lambda_k \in \{19,36,51,64,75,84,91,96,99,100\}$. Hence, the problem is very stiff.
}
\label{fig:Fisher_100}
\end{figure}

\subsection{Swift-Hohenberg equation} \label{sec:application_SH}

Swift-Hohenberg's (SH) equation 
\begin{equation} \label{eq:SH}
u_t = (\alpha-1)u - 2 u_{xx} - u_{xxxx} - u^3, \quad \alpha \in \R
\end{equation}
is used as a model for pattern formation due to a finite wavelength instability, such as in Rayleigh–B\'enard convection \cite{swift-hohenberg,cross_hohenberg}. Considering even boundary conditions leads (via the cosine Fourier expansion \eqref{eq:cosine_Fourier_expansion} plugged in \eqref{eq:SH}) to
\begin{equation} \label{eq:ODEs_SH}
\frac{d \ta_k}{dt} = f_k^{({\rm SH})}(\ta) \bydef \left( -k^4+2k^2+ \alpha-1 \right) \ta_k - (\ta*\ta*\ta)_k, \quad k \ge 0.
\end{equation}
Recalling \eqref{eq:ODEs_general}, we get for SH that $\lambda_k = -k^4+2k^2+ \alpha-1$ and $Q_{k}(a) = - (a^3)_k$. 

\subsubsection{The bound \boldmath$\gamma(r)~$\unboldmath and \boldmath$h_{max}$\unboldmath}

Given any $c \in \ball{r}{\ba}$ and $h \in \ball{1}{0}$, $D_a \cQ(c) h = - 3c^2*h$, and hence
\[
\| D_a \cQ(c) \|_{B(X_{\nu,1})} \le \gamma(r) \bydef
3 \left( \| \ba\|_{X_{\nu,1}} + r \right)^2.
\]
We therefore set
\begin{equation} \label{eq:Z(r)_SH}
Z(r) = 3 h \delta
\| \ba\|_{X_{\nu,1}}^2 + 
6h \delta \| \ba\|_{X_{\nu,1}} r + 3h \delta
r^2
\end{equation}
and 
\begin{equation} \label{eq:Z1_SH}
Z_1 = Z(0) = 3 h \delta
\| \ba\|_{X_{\nu,1}}^2 .
\end{equation}
Recall \eqref{eq:h_max} and note that for SH,
\[
h < h_{max} = \frac{1}{ 3 \delta \| \ba\|_{X_{\nu,1}}^2}.
\]

We consider $\alpha = 8.1$. At that parameter value, there are $2$ unstable eigenvalues: $7.1$ and $8.1$. Fix the initial condition to be  $u_0(x) = 0.02 \cos(x)$ which is roughly is the unstable manifold of $u \equiv 0$. We fix $\hat k =5$, which fixes the number of blocks $\cL_k$ ($k=0,\dots,5$), that we invert using the computer-assisted approach of Section~\ref{sec:bounds_cL_k_inv_small_k}. For the whole integration, we fixed the number of Fourier coefficients to be $15$. We report the results in Table~\ref{Tab:SH} and in Figure~\ref{fig:SH_8.1_plusSS}.

\begin{table}[ht]
\centering
{
\begin{tabular}{ccccccc}
\hline
Steps & $h$ & $\#$ of Cheb. coeff. & $\delta$ & $r_0$ \\
\hline\\[-3mm]
$1$ & $1.3391 \times 10^{-1}$ & $17$ & $2.9986$ & $7.3026 \times 10^{-16}$ \\
$2$ & $2.0136 \times 10^{-1}$ & $20$ & $5.2733$ & $2.4007 \times 10^{-13}$ \\
$3$ & $9.9226 \times 10^{-2}$ & $19$ & $2.2522$ & $1.4644 \times 10^{-12}$ \\
$4$ & $5.8592 \times 10^{-2}$ & $17$ & $1.6137$ & $4.6899 \times 10^{-12}$ \\
$5$ & $3.8443 \times 10^{-2}$ & $16$ & $1.455$ & $1.1851 \times 10^{-11}$ \\
$6$ & $2.5729 \times 10^{-2}$ & $15$ & $1.455$ & $2.6352 \times 10^{-11}$ \\
$7$ & $2.0841 \times 10^{-2}$ & $14$ & $1.455$ & $5.7253 \times 10^{-11}$ \\
$8$ & $1.6881 \times 10^{-2}$ & $14$ & $1.455$ & $1.2073 \times 10^{-10}$ \\
$9$ & $1.4371 \times 10^{-2}$ & $13$ & $1.455$ & $2.5096 \times 10^{-10}$ \\
$10$ & $1.2934 \times 10^{-2}$ & $13$ & $1.455$ & $5.2304 \times 10^{-10}$ \\
$15$ & $9.2809 \times 10^{-3}$ & $13$ & $1.455$ & $1.9945 \times 10^{-8}$ \\
$20$ & $7.2538 \times 10^{-3}$ & $12$ & $1.455$ & $7.5762 \times 10^{-7}$ \\
$25$ & $6.7263 \times 10^{-3}$ & $12$ & $1.455$ & $2.9086 \times 10^{-5}$ \\
$30$ & $6.2371 \times 10^{-3}$ & $11$ & $1.455$ & $1.1231 \times 10^{-3}$ \\
$35$ & $6.2371 \times 10^{-3}$ & $11$ & $1.455$ & $4.6083 \times 10^{-2}$ \\
\hline 
\end{tabular}
}
\caption{Data for the rigorous enclosure of the solution of the Cauchy problem for SH equation.}
\label{Tab:SH}
\end{table}

We report some results in Figure~\ref{fig:SH_8.1_plusSS}.

\begin{figure}[h!]
\begin{center}
\includegraphics[width=8.3cm]{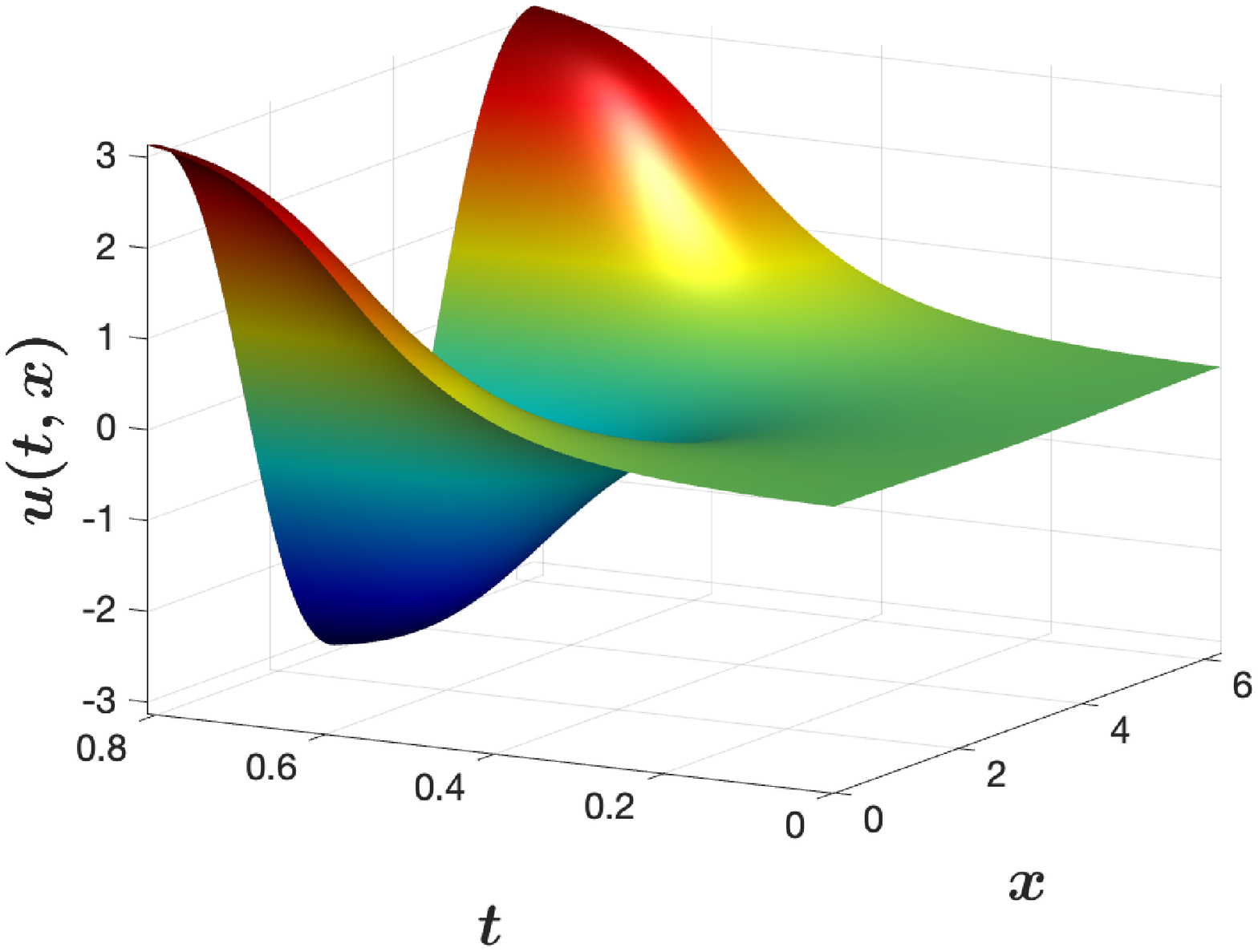}
\includegraphics[width=8cm]{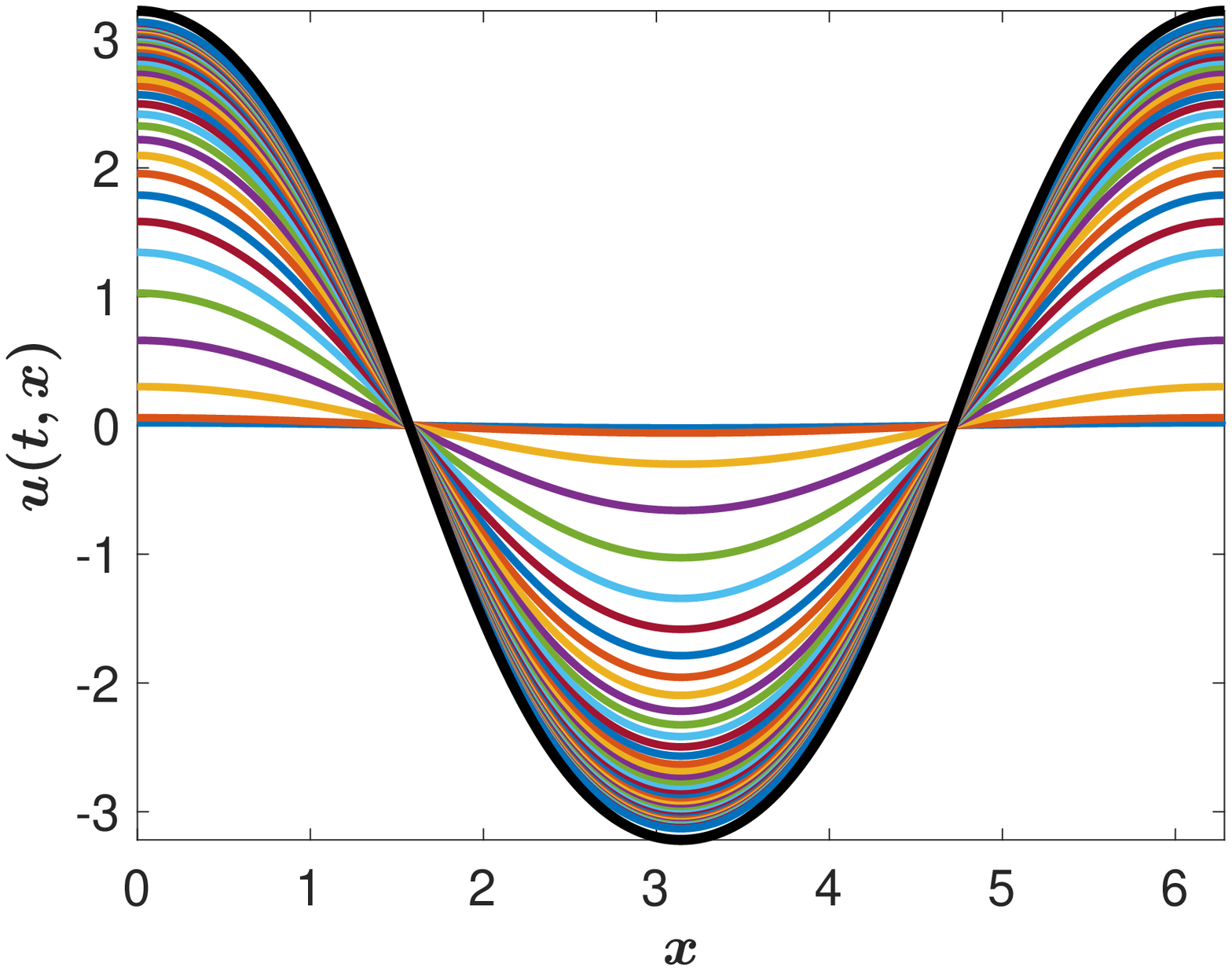}
\end{center}
\vspace{-.3cm}
\caption{The solution of the Cauchy problem for the Swift-Hohenberg equation with $\alpha = 8.1$ and $u_0(x) = 0.02 \cos(x)$, which is roughly is the unstable manifold of $u \equiv 0$. The number of steps is $35$ and the total integration time is $0.81035$. In thick black, we portrait the graph of a steady states of \eqref{eq:SH}, which shows that we almost an entire connecting orbit between $0$ and the nontrivial steady states. Note that there are two unstable eigenvalues: $7.1$ and $8.1$.
}
\label{fig:SH_8.1_plusSS}
\end{figure}

\section{Future directions} \label{sec:future}

There are many future research directions and open problems related to the described method that we will pursue in our future work. Major research efforts will be devoted to making our approach applicable for performing computer-assisted proofs in dynamics that require validated forward integration, like the existence of connecting orbits. To achieve this ultimate goal several improvements will be introduced. First, to deal with longer orbits (i.e. larger $h$) and solutions with larger norms, $\cL = DF(\ba)$ should be considered instead of $\cL = DF(0)$ when computing an approximate inverse for a Newton-like operator. Second, an effective way of fighting the wrapping effect needs to be employed (see Remark~\ref{rem:wrapping_effect}). Virtually all forward integration schemes suffer the issue when the error resulting from a single step of forward integration is being accumulated in a multiplicative way and leads to blow-up of bounds after a finite time. A promising simple solution to this issue in the context of our method is to employ a multi-step forward integration operator instead of the single-step one. 

A very important research direction is to adapt our technique to other important PDEs that are beyond scope of the present implementation, including Kuramoto-Sivashinsky,
Burgers, Navier-Stokes, Cahn-Hilliard and Ohta-Kawasaki model.  Generalizing this approach to PDEs defined on higher dimensional domains (allowing a Fourier expansion in space) is also an interesting direction. 

\section{Acknowledgements}
The project was initiated at  Semester Program Workshop
``Computation in Dynamics'' organized at the Institute for Computational and Experimental Research in Mathematics (ICERM) in 2016. JC was partially supported by NAWA Polish Returns grant PPN/PPO/2018/1/00029. JPL was supported by an NSERC Discovery Grant.

\bibliographystyle{unsrt}
\bibliography{papers}

\appendix
\input{includes/appendix}
\end{document}

%% file: includes/linear-operators.tex
In this section, we derive a uniform bound for $\|\cL_k^{-1}\|_{B(\ell^1)}$ for large $k$, first by obtaining a uniform bound for $\|(\cL_k^{(N)})^{-1}\|_{B(\ell^1)}$ for large $k$ and $N$ (see Corollary~\ref{cor:LkinvProj}), and then by using the result of Lemma~\ref{lem:bounding_invLk_small_k}.

\paragraph{Notation.} 
For a square matrix $A$ by $A_{\co{i}}$ we denote the $i$-th column vector of $A$, by $A_{\ro{j}}$ we denote the $j$-th row vector of $A$, by $A_{\co{i:j}}$ we denote the matrix composed out of the $i,\dots,j$-th columns of $A$, by $A_{\ro{i:j}}$ we denote the matrix composed out of the $i,\dots,j$-th columns of $A$.

As before, we denote the $N+1$ dimensional truncated linear operator $\cL_k$ by $\cLp_k$. $\cLp_k$ acts on vectors of the Chebyshev coefficients corresponding to $T_0,\dots,T_{N}$. We decompose $\cLp_{k}$ as the sum of the tridiagonal operator and the single non-zero row operator, i.e.,
\begin{equation}
\label{plusrone}
  \cLp_{k} \bydef M_k +  U  \bydef \left[ \begin{array}{cccc} 1&0&\cdots&0\\\eigk&\ &\ &\ \\0&\ &T_k&\ \\\vdots &\ &\ &\ \end{array}\right] + 
\left[ \begin{array}{cccc}0&-2&2&\dots\\\ &\ &\ &\ \\\ &\ &0&\ \\\ &\ &\ &\ \\ \end{array} \right],
\end{equation}
where 
\begin{equation}
\label{tridiag}
T_{k} \bydef \left[ \begin{array}{cccc} 
 2 &-\eigk&0&\cdots\\
\eigk& 4 &-\eigk&0\\
\ddots&\ddots&\ddots&\ddots\\
\dots&0&\eigk& 2N
\end{array}\right],
\end{equation}
$T_k$ diagonal elements are denoted by 
\[
d_k\bydef 2k\text{ for }k\geq 1.
\]

We derive a uniform bound for $\|\cLpkinv\|_{B(\ell^1)}$ for large $N$ and $k$. For the purpose of this section, we fix the  approximation dimension $N$ and the index $k$. 

Within this section, by $\left\|\cdot\right\|_{\ell^1}$ we denote the analogue of \eqref{eq:ell_one_norm} for finite vectors, and the $B(\ell^1)$ norm for a finite matrix $A$ takes the form
\begin{align}
\|v\|_{\ell^1} &\bydef v_1 + \sum_{j=2}^{N}{2|v_j|},\label{eq:vell1}\\
\| A\|_{B(\ell^1)} &\bydef \max \left\{|a_{1,1}|+2\sum_{j=2}^{N}|a_{j,1}|,\ \sup_{n > 1} \frac{1}{2}|a_{1,n}| + \sum_{j=2}^{N}|a_{j,n}|  \right\}.\label{eq:Bell1matrix}
\end{align}

To distinguish the vector $\ell^1$ norm from the standard matrix norm, by $\left\|
\cdot\right\|_{\ellmat}$ we denote the sum of absolute values  norm, i.e. for a vector / matrix $v,A$ we have $\left\|v\right\|_{\ellmat} = \sum_{j=1}^{N}{\left|v_{j}\right|}$, 
$\left\|
A\right\|_{\ellmat} = \max_{i}{\sum_{j=1}^{N}{\left|A_{ij}\right|}}$.

First, we are concerned with bounding the norm of the inverse of the tri-diagonal part of $\cLpk$, excluding the first row and column (i.e., the operator $T_k$). Second, using the Sherman-Morrison formula, we eventually bound $\|\cLpkinv\|_{B(\ell^1)}$ by simply $2\|T_k^{-1}\|_1$.
\newcommand{\mR}{\mathbb{R}}
\newcommand{\mnorm}{\ell^1}
%
%
\subsubsection{Finite inverse tridiagonal operator \boldmath$\left\|T_k^{-1}\right\|_{\ellmat}$\unboldmath norm uniform with respect to \boldmath$N$\unboldmath}

\paragraph{Notation.} Let $N>0$ be an even integer defining the projection size (fixed), $k>1$. Given sequence $\{x_j\}_{j=1}^N$ by
\begin{equation}
\label{tridiagser}
T_k(x_1, \dots, x_N) \bydef
\begin{bmatrix}
   x_1 &-\eigk&0&\dots&\ \\\eigk& x_2 &-\eigk&0&\dots\\\ &\ddots&\ddots&\ddots&\ \\\dots&0&\eigk& x_{N-1} &-\eigk\\\ &\dots&0&\eigk& x_N
\end{bmatrix},
\end{equation}
we denote the $N\times N$ tridiagonal matrix with 
the sequence on the diagonal, and $\pm \mu_k$ under and over-diagonal respectively.

We assume that $k$ denoting the Fourier coefficient is fixed, and denote $T = T_k$. By default $T$ without parentheses denotes the matrix with the sequence $d_k \bydef 2k$ as its diagonal for $k\geq 1$. We also use the notation $T^{-1}(x_1,\dots,x_N)$ to denote the inverse of $T(x_1,\dots,x_N)$. Observe that any matrix having the tridiagonal form like \eqref{tridiagser} is invertible when $\mu_k\neq 0$ due to an easy determinant computation (demonstrated in \cite{cyranka_mucha}). If not specified otherwise, we assume that $\mu_k>0$.
\begin{remark}
We base our large tridiagonal matrices analysis on a kind of \emph{divide and conquer} paradigm; that is,  in order to invert a large 
tridiagonal matrix, we decompose it recursively into smaller and smaller blocks. The blocks appearing in the inverse 
have a simple and explicit form. In Lemma~\ref{lemblock} we formalize this observation, to which we are going to refer often in the 
forthcoming analysis. The formulas that we used are a special case of a formula known in the literature as Banachiewicz inversion formula based on Schur complement \cite{schur}. The same formula in the context of the validated numerical method was used in \cite{nakao_shur}.
\end{remark}

The presented analysis was used already in \cite{cyranka_mucha}, where it is shown that $\|T_k^{-1}\|_{\ellmat}$ is bounded independently of $\mu$.

\begin{theorem}[Theorem~4.11 in \cite{cyranka_mucha}, setting $m=1$]
\label{tridiag}
    For any $\mu_k\in\mR$. There exists a constant that bounds $\|T_k^{-1}\|_{\ellmat}$ independently of $\mu_k$ for all $k$ and $N$.
\end{theorem}
However, in the present paper, we make the results stronger by not only showing that $\|T_k^{-1}\|_{\ellmat}$ is bounded uniformly in $k$ and $N$ but also derive an explicit upper bound for it. We provide a new theoretical framework building upon the auxiliary lemmas from \cite{cyranka_mucha} and eventually prove the main theorem.

\paragraph{Notation} Let $\upmu_k$ denote the $n\times(N-n)$ matrix having single non-zero element in the lower left corner
\begin{equation}
\label{eq:muk}
\upmu_k = \begin{bmatrix} 0 & & & \\ \vdots & &\ddots& \\ 0&0& & \\\mu_k&0&\dots&0\end{bmatrix},\ \text{then }
\upmu_k\upmu_k^T = \begin{bmatrix}  & & &0\\ &\ddots& &\vdots\\ & &0&0\\0&\dots&0&\mu_k^2\end{bmatrix}\in M_n(\R).
\end{equation}
\newcommand{\inv}{\mbox{Inv}}
\newcommand{\id}{\mbox{Id}}
\begin{lemma}
  \label{lemblock}
Let $0 < n < N$ be even. Let $T\in\mathbb{R}^{N\times N}$ be a tridiagonal matrix of the form \eqref{tridiagser} with an arbitrary sequence on the diagonal. $T$ is decomposed
in the following way
\[
T = 
\begin{bmatrix}
F & -\upmu_k \\
\upmu_k^T & R
\end{bmatrix},
\]
where $F\in M_n(\R)$, and $R\in M_{N-n}(\R)$. Then it holds that
\[
T^{-1} = \begin{bmatrix} F & -\upmu_k \\ \upmu_k^T & R\end{bmatrix}^{-1} = 
\begin{bmatrix} \left(F + \upmu_k\upmu_k^TR^{-1}_{11}\right)^{-1} & F^{-1}\upmu_k\left(R + \upmu_k^T\upmu_kF^{-1}_{nn}\right)^{-1} \\ -R^{-1}\upmu_k^T\left(F + \upmu_k\upmu_k^T R^{-1}_{11} \right)^{-1} & \left(R + \upmu_k^T\upmu_kF^{-1}_{nn}\right)^{-1}\end{bmatrix}.
\]

This formula is a special case of a formula known in the literature as Banachiewicz inversion formula based on Schur complement \cite{schur}.
\end{lemma}
\begin{proof}
We want to compute 
\[
T^{-1} = \begin{bmatrix} \inv_{11} & \inv_{12} \\ \inv_{21} & \inv_{22} \end{bmatrix},
\]
where $\inv_{11}, \inv_{12}, \inv_{21}, \inv_{22}$ are the blocks composing $T^{-1}$ of appropriate dimensions. 

Directly from the condition for the inverse matrix it follows that
\begin{equation*}
\begin{array}{cc}
F\inv_{11} - \mu_k\inv_{21} = I,&\upmu_k^T\inv_{11} + R\inv_{21} = 0,\\
F\inv_{12}-\upmu_k\inv_{22} = 0,&\upmu_k^T\inv_{12} + R\inv_{22} = I.
\end{array}
\end{equation*}
The presented formula for $\inv_{ij}$ can be obtained by noting.
\[
\upmu_kR^{-1}\upmu_k^T = \upmu_k\upmu_k^T(R^{-1}_{11}),\qquad \upmu_k^TF^{-1}\upmu_k = \upmu_k^T\upmu_k(F^{-1}_{nn}),
\]
and decoupling the solutions to the equations for $\inv_{11}, \inv_{12}, \inv_{21}, \inv_{22}$.
\end{proof}
\bigskip

\begin{definition}
\label{defaahat}
Let $\{d_j\}_{j=1}^N$ be the sequence of $T_k$ diagonal elements given by $d_j = 2j$, and let us define the following recursive sequences 
\begin{equation*}
\arraycolsep=2pt\def\arraystretch{1.4}
\begin{array}{lll}
a_j &=a_j(a_{j-1}) \bydef
\displaystyle
\frac{ d_{N-2j+2} + a_{j-1}\mu_k^2}{ d_{N-2j+1}d_{N-2j+2} + a_{j-1}d_{N-2j+1}\mu_k^2 + \mu_k^2}
\vspace{.3cm}
\\
\hat{a}_j &=\hat{a}_j(a_{j-1}) \bydef 
\displaystyle
\frac{ \mu_k }{ d_{N-2j+1}d_{N-2j+2} + a_{j-1}d_{N-2j+1}\mu_k^2 + \mu_k^2}
\vspace{.3cm}
\\
b_j &=b_j(b_{j-1}) \bydef
\displaystyle
\frac{ d_{2j-1} + b_{j-1}\mu_k^2}{ d_{2j}d_{2j-1}  + b_{j-1}d_{2j}\mu_k^2+\mu_k^2}
\end{array}
\end{equation*}
for $j=1,2,3,\dots,\frac{N}{2}$, with $a_0,\hat{a}_0,b_0 = 0$.
\end{definition}

\begin{lemma}
  \label{lemmono}
  Let $\{a_j\}_{j=1}^{\frac{N}{2}}$, $\{\hat{a}_j\}_{j=1}^{\frac{N}{2}}$, $\{b_j\}_{j=1}^{\frac{N}{2}}$ be the recursive sequences defined in Definition~\ref{defaahat}. It holds that
  \[
    a_j(a_{j-1}),\ b_j(b_{j-1})
  \]
  i.e. $a_j$ as a function of $a_{j-1}$ and $b_j$ as a function of $b_{j-1}$ are increasing $\text{ for }j=2,\dots,\frac{N}{2}$. And
  \[
    \hat{a}_j(a_{j-1})
  \]
  i.e. $\hat{a}_j$ as a function of $a_{j-1}$ is decreasing $\text{ for }j=2,\dots,\frac{N}{2}$.
\end{lemma}
\begin{proof}
  We treat $a_{j-1}$, $b_{j-1}$ as (continuous) variables in the formulas given in Definition~\ref{defaahat}, and compute
  \begin{align*}
    \frac{d\,a_j}{d\,a_{j-1}} &= \frac{\mu_k^4}{\left( d_{N-2j+1}d_{N-2j+2} + a_{j-1}d_{N-2j+1}\mu_k^2 + \mu_k^2\right)^2} > 0,\\
    \frac{d\,b_j}{d\,b_{j-1}} &= \frac{\mu_k^4}{\left( d_{2j}d_{2j-1} + b_{j-1}d_{2j}\mu_k^2+\mu_k^2\right)^2} > 0,\\
    \frac{d\,\hat{a}_j}{d\,a_{j-1}} &< 0. \qedhere
  \end{align*}
  
\end{proof}
\begin{lemma}
  \label{lema}
  Let $\{a_j\}_{j=1}^{\frac{N}{2}}$, $\{\hat{a}_j\}_{j=1}^{\frac{N}{2}}$, $\{b_j\}_{j=1}^{\frac{N}{2}}$ be the sequences defined in Definition~\ref{defaahat}. It holds that
  \[
    0 < a_j < \frac{\sqrt{2}}{\mu_k},\qquad 0 < \hat{a}_j < \frac{1}{\mu_k}, \qquad 0<b_j < \frac{1}{\mu_k}
  \]
  for all $j=1,\dots,\frac{N}{2}$.
\end{lemma}
\begin{proof}
The bound $a_j>0$ is obvious. We proceed by induction. First, we show that $a_1=\frac{ d_N }{ d_{N-1}d_N +\mu_k^2} < \frac{\alpha}{\mu_k}$, 
which is satisfied for $\alpha\geq 1$, as $0< \alpha d_{N-1}d_N  + \alpha\mu_k^2- d_N\mu_k$. We find $\alpha\geq 1$, such that assuming $a_{j-1} < \frac{\alpha}{\mu_k}$, we show $a_j<\frac{\alpha}{\mu_k}$. Using the monotonicity property from Lemma~\ref{lemmono} we plug in $a_{j-1} = \frac{\alpha}{\mu_k}$ in the formula for $a_j(a_{j-1})$ (Def.~\ref{defaahat}) and find minimal $\alpha$ satisfying the inequality
\[
    a_j < a_j\left(\frac{\alpha}{\mu_k}\right) < \frac{ d_{N-2j+2} + \alpha\mu_k}{ d_{N-2j+1}d_{N-2j+2} + \alpha d_{N-2j+1}\mu_k + \mu_k^2} < \frac{\alpha}{\mu_k},
\] 
which is satisfied when
$0 < \alpha d_{N-2j+1}d_{N-2j+2} + \mu_k(\alpha^2 d_{N-2j+1}-d_{N-2j+2})$,
  and hence
  \begin{equation}
    \label{eqineq2}
    0 < \alpha^2 d_{N-2j+1}-d_{N-2j+2} = \alpha^22(N-2j+1) - 2(N-2j+2).
  \end{equation}
  By analyzing the worst case $(j = \frac{N}{2})$ \eqref{eqineq2} holds for $\alpha \geq \sqrt{2}$. The minimal value of $\alpha$ such that $a_j < \frac{\alpha}{\mu_k}$ is $\alpha = \sqrt{2}$. 
  
  In order to show the second inequality, first notice that
$\hat{a}_1 = \frac{\mu_k}{ d_{N-1}d_N  + \mu_k^2} = \frac{\mu_k}{d_{N-1}d_N + \mu_k^2} < \frac{1}{\mu_k}\text{ for }N> 1$.
  By induction we show that $0< \hat{a}_j< \frac{1}{\mu_k}$. Due to the monotonicity property from Lemma~\ref{lemmono}, we plug in $a_{j-1} = 0$ in the formula for $\hat a_j$ and obtain
  \[
    0 < \hat{a}_j < \hat{a}_j(0)= \frac{ \mu_k }{ d_{N-2j+1}d_{N-2j+2} +  \mu_k^2} < \frac{1}{\mu_k}\text{ for }N> 1.
  \]
Now let us turn into the third inequality, we again proceed by induction. $b_1 = \frac{d_1}{d_2d_1 + \mu_k^2} < \frac{1}{\mu_k}$ is easy to verify. We show that $0 < b_j < \frac{1}{\mu_k}$. Using the monotonicity property from Lemma~\ref{lemmono} we plug in $b_{\frac{N-2j-2}{2}} = \frac{1}{\mu_k}$ in the formula for $b_{\frac{N-2j}{2}}$ and obtain
\[
0 < b_{\frac{N-2j}{2}} < b_{\frac{N-2j}{2}}\left(\frac{1}{\mu_k}\right) < \frac{d_{N-2j+1} + \mu}{d_{N-2j+1}d_{N-2j+2} + d_{N-2j+2}\mu + \mu^2} < \frac{1}{\mu_k},
\]
satisfied when
$0 < d_{N-2j+1}d_{N-2j+2} + (d_{N-2j+2}-d_{N-2j+1})\mu$,
which is clearly true for all $j=1,\dots,\frac{N}{2}$.
\end{proof}
%
%
%
%
%
%
%
%
%
%
\begin{lemma}
  \label{lema2}
  Let $\{a_j\}_{j=1}^{\frac{N}{2}}$, $\{\hat{a}_j\}_{j=1}^{\frac{N}{2}}$  be the sequences defined in Definition~\ref{defaahat}.
  Also let $N > 2^p\mu_k$ and even, $p>0$. It holds that
  \[
    a_j \leq \frac{1}{2^{p+1}\mu_k},\ b_{\frac{N}{2}-j+1}\leq \frac{1}{2^{p+1}\mu_k}\text{, and }\hat{a}_j \leq \frac{1}{(2^{2p+2} +1)\mu_k} 
  \]
  for all $j$ such that 
  \begin{equation}
    \label{muineq}
    N-2j+1 \geq 2^p\mu_k\qquad \left(\text{i.e., for all} ~ j \leq \frac{N}{2} - \mu_k + 1\right).
    \end{equation}
\end{lemma}
\begin{proof}
Using the upper bound of Lemma~\ref{lema}, we put $a_{j-1}=\frac{\sqrt{2}}{\mu_k}$ in the formula $a_{j}(a_{j-1})$ (which increases in $a_{j-1}$), and denote $\beta \bydef d_{N-2j+1} = 2(N-2j+1) = 2N-4j+2$, then $d_{N-2j+2} = \beta+2$. We verify that
  \[
  a_{j}(\frac{\sqrt{2}}{\mu_k}) = 
    \frac{ \beta+2 + \sqrt{2}\mu_k}{ \beta(\beta+2) + \sqrt{2}\beta \mu_k + \mu_k^2} < \frac{1}{2^{p+1}\mu_k},
  \]
  which reduces to
  \begin{equation}
    \label{muineq2}
    g(\beta) \bydef \beta(\beta+2) + \sqrt{2}\beta\mu_k + \mu_k^2 - 2^{p+1}\beta\mu_k - 2^{p+2}\mu_k - 2^{p+1}\sqrt{2}\mu_k^2 > 0.
  \end{equation}
  The function $g(\beta)$ is increasing for all $\beta \ge \left( 2^p - \frac{\sqrt{2}}{2}\right) \mu_k-1$, that is for all $j$ such that  $2(N-2j+1) \ge \left( 2^p - \frac{\sqrt{2}}{2}\right) \mu_k-1$, a condition which is ensured by assumption \eqref{muineq}.
  Plugging $\beta=2^{p+1}\mu_k \ge \left( 2^p - \frac{\sqrt{2}}{2}\right) \mu_k-1$ in \eqref{muineq2} leads to \[
  0<\mu_k^2 = g(2^{p+1}\mu_k) \le g(\beta)\]
  for all $\beta \geq 2^{p+1}\mu_k$, that is for all $j$ such that
  $2(N-2j+1) \geq 2^{p+1}\mu_k$, which follows by \eqref{muineq}. 
  
  Analogous computation shows the bound $b_{\frac{N}{2}-j+1}\leq \frac{1}{2^{p+1}\mu_k}$, using the upper bound of Lemma~\ref{lema}, and putting $b_{\frac{N}{2}-j}=\frac{1}{\mu_k}$ in the formula $b_{\frac{N}{2}-j+1}(b_{\frac{N}{2}-j})$. Regarding the bound for $\hat{a}_j$ we have
  \begin{multline*}
  \hat{a}_j = \frac{ \mu_k }{ d_{N-2j+1}d_{N-2j+2} + a_{j-1}d_{N-2j+1}\mu_k^2 + \mu_k^2} \leq \frac{\mu_k}{\beta(\beta + 2) + \mu_k^2} 
  <\\
  \frac{\mu_k}{2^{p+1}\mu_k(2^{p+1}\mu_k+2) + \mu_k^2} = 
  \frac{1}{(2^{2p+2}+1) \mu_k + 2^{p+2}}. \qedhere
  \end{multline*}
\end{proof}

Our eventual estimate for $T^{-1}$ relies on computing all of its components explicitly and then bounding the resulting sum.
The explicit formulas for $T^{-1}$ are build using the recursive formulas presented in Definition~\ref{defaahat}.
\begin{lemma}
\label{lemblocks}
Let $n < \frac{N}{2}$. The $N-2n\times N-2n$ top left corner submatrix of $T^{-1}$ is given explicitly by
\[
T^{-1}(d_1, d_2, \dots, d_{N-2n} + \mu_k^2a_n),
\]
whereas the $2n\times 2n$ bottom right corner  submatrix of $T^{-1}$ is given explicitly by
\[
T^{-1}(d_{N-2n+1} + \mu_k^2b_{\frac{N-2n}{2}}, d_{N-2n+2},\dots, d_{N}),
\]
where $a_j$ and $b_{\frac{N-2j-2}{2}}$ are elements of the recursive series from Definition~\ref{defaahat}.
\end{lemma}
\begin{proof}
Follows from Lemma~\ref{lemblock}
with $R=T\left(d_1, d_2, \dots, d_{N-2n}\right)$, and $F=T(d_{N-2n+1}, d_{N-2n+2},\dots, d_{N})$, and computing the recursion for the elements $R^{-1}_{11}$ and $F^{-1}_{nn}$. For the detailed proof refer to \cite{cyranka_mucha}.
\end{proof}

We have the following corollary from Lemma~\ref{lemblocks}. In the proofs in the sequel we will often use the explicit form of $T_k^{-1}$ diagonal blocks given by
\begin{corollary}
\label{corblocks}
The $2\times 2$ dimensional diagonal blocks of $T^{-1}$ are given by
\begin{equation}
\label{eq:I}
\widetilde{I}\bydef 
  \begin{bmatrix}
    d_{N-2n+1}+\mu_k^2b_{\frac{N-2n}{2}} & \mu_k \\ -\mu_k & d_{N-2n+2} +\mu_k^2a_{n-1}       
  \end{bmatrix}^{-1} = \frac{1}{D}\begin{bmatrix}
d_{N-2n+2} + \mu_k^2 a_{n - 1} & \mu_k \\
 -\mu_k & d_{N-2n+1} + \mu_k^2b_{\frac{N-2n}{2}}
\end{bmatrix}
\end{equation}
for $n=1,\dots,\frac{N}{2}$, where $a_{n-1}$ and $b_{\frac{N-2n}{2}}$ are elements of the recursive series from Definition~\ref{defaahat}, $D=(d_{N-2n+2} + \mu_k^2 a_{n - 1})(d_{N-2n+1} + \mu_k^2b_{\frac{N-2n}{2}}) + \mu_k^2$ is the determinant of $\widetilde{I}$.

The $n$-th diagonal block above corresponds to \[
T^{-1}_{\ro{N-2n+1:N-2n+2},\co{N-2n+1:N-2n+2}}.\] 
In particular $1$-st block denotes the bottom-right diagonal block of $T^{-1}$.
\end{corollary}
\begin{theorem}
\label{thm:explicit}
Let $\widetilde{I}$ be the block \eqref{eq:I}. Consider the $2n\times 2n$ dimensional bottom right corner square submatrix of $T^{-1}$ 
\[
\widetilde{T}^{-1} \bydef T^{-1}(d_{N-2n+1} + \mu_k^2b_{\frac{N-2n}{2}},\ d_{N-2n+2},\ \dots,\ d_{N}).
\]
The following recursive explicit formulas hold
\begin{align*}
    \widetilde{T}^{-1}_{\ro{1:2},\co{1:2}} &= \widetilde{I},\\
    \widetilde{T}^{-1}_{\ro{3:4},\co{1:2}} &= \mu_k\begin{bmatrix}-a_{n-1}\\\hat{a}_{n-1}\end{bmatrix}\widetilde{T}^{-1}_{\ro{2},\co{1:2}},\\
    \vdots &\\
    \widetilde{T}^{-1}_{\ro{2j+1:2j+2},\co{1:2}} &= \mu_k\begin{bmatrix}-a_{n-j}\\\hat{a}_{n-j}\end{bmatrix}\widetilde{T}^{-1}_{\ro{2j},\co{1:2}}. 
\end{align*}
\end{theorem}
\begin{proof}
  We decompose into blocks
  \[
    \widetilde{T}^{-1} = \left[ \begin{matrix} T(d_{N-2n+1} + \mu_k^2b_{\frac{N-2n}{2}},\ d_{N-2n+2},\dots,d_{N-2}) & -\upmu_k \\ \upmu_k^T & T(d_{N-1},d_N) \end{matrix}\right]^{-1} = \begin{bmatrix} \inv_{11}^{(1)} & \inv_{12}^{(1)}\\ \inv_{21}^{(1)} & \inv_{22}^{(1)}\end{bmatrix}.
  \]
  
  We apply Lemma~\ref{lemblock} with $F = T(d_{N-2n+1} + \mu_k^2b_{\frac{N-2n}{2}},\ d_{N-2n+2},\dots,d_{N-2})$, and $R = T(d_{N-1},d_N)$. It follows that
  \begin{multline*}
    \inv_{11}^{(1)} = \left( T(d_{N-2n+1} + \mu_k^2b_{\frac{N-2n}{2}},\ d_{N-2n+2},\dots,d_{N-2}) + \upmu_k\upmu_k^T T^{-1}_{11}(d_{N-1},d_N)\right)^{-1} =\\ T^{-1}(d_{N-2n+1} + \mu_k^2b_{\frac{N-2n}{2}},\ d_{N-2n+2},\dots,d_{N-2} + \mu_k^2a_1).
  \end{multline*}
  (observe that $\upmu_k$ is $(N-2)\times 2$ dimensional matrix -- set $n=N-2$ in \eqref{eq:muk}). It holds that
  \[
    \inv_{21}^{(1)} = - T^{-1}(d_{N-1}, d_N)\upmu_k^T\inv_{11}^{(1)},\qquad (\inv_{21}^{(1)})_{\co{1:2}} = \mu_k\begin{bmatrix} -a_1\\\hat{a}_1\end{bmatrix}(\inv_{11}^{(1)})_{\ro{N-2},\co{1:2}}.
  \]
  In order to compute $(\inv_{11}^{(1)})_{\ro{N-2},\co{1:2}}$ we decompose $\inv_{11}^{(1)}$ into blocks further
  \[
    \inv_{11}^{(1)} = \begin{bmatrix} \inv_{11}^{(2)} & \inv_{12}^{(2)}\\ \inv_{21}^{(2)} & \inv_{22}^{(2)} \end{bmatrix} = \begin{bmatrix} T(d_{N-2n+1} + \mu_k^2b_{\frac{N-2n}{2}},\ d_{N-2n+2},\dots,d_{N-4}) & -\upmu_k \\ \upmu_k^T & T(d_{N-3}, d_{N-2}) + \mu_k^2a_1)\end{bmatrix}^{-1}.
  \]
  It holds that
  \[
    \inv_{21}^{(2)} = -T^{-1}(d_{N-3}, d_{N-2}+\mu_k^2a_1)\upmu_k^T\inv_{11}^{(2)}.
  \]
  Hence
  \[
    (\inv_{11}^{(1)})_{\ro{N-3:N-2},\co{1:2}} = (\inv_{21}^{(2)})_{\ro{1:2},\co{1:2}} = \mu_k\begin{bmatrix}-a_2\\ \hat{a}_2\end{bmatrix}(\inv_{11}^{(2)})_{\ro{N-4},\co{1:2}}.
  \]
  Repeating this procedure, we find
  \[
    (\inv_{11}^{(2)})_{\ro{N-5:N-4},\co{1:2}} = (\inv_{21}^{(3)})_{\ro{1:2},\co{1:2}} = \mu_k\begin{bmatrix}-a_3\\\hat{a}_3\end{bmatrix}(\inv_{11}^{(3)})_{\ro{N-6},\co{1:2}},
  \]
  and
  \[
    (\inv_{11}^{(j)})_{\ro{N-2j-1:N-2j},\co{1:2}} = (\inv_{21}^{(j-1)})_{\ro{1:2},\co{1:2}} = \mu_k\begin{bmatrix}-a_{j+1}\\\hat{a}_{j+1}\end{bmatrix}(\inv_{11}^{(j-1)})_{\ro{N-2j-2},\co{1:2}},
  \]  
  and finally (after several steps of recursion)
  \begin{gather}
    \inv_{21}^{(n-1)} = \mu_k\begin{bmatrix}-a_{n-1}\\\hat{a}_{n-1}\end{bmatrix}\left(\inv_{11}^{(n-1)}\right)_{\ro{2}}\nonumber\\
    \inv_{11}^{(n-1)} = \widetilde{I} = T^{-1}(d_{N-2n+1} + \mu_k^2b_{\frac{N-2n}{2}},\ d_{N-2n+2} + \mu_k^2 a_{n-1}).\label{recurend}
  \end{gather}
We emphasize that the recursion is being finished here as $\inv_{11}^{(n-1)}$ is the $2\times 2$ upper left corner matrix of $T^{-1}$. 
Observe that $T^{-1}_{\ro{N-2j-1:N-2j}, \co{1:2}} = (\inv_{11}^{(j)})_{\ro{N-2j-1:N-2j},\co{1:2}}\text{ for $j=1,\dots,n-1$}$. Therefore we obtained the claim.
\end{proof}
\begin{remark}
By repetitive application of Theorem~\ref{thm:explicit} we obtain the explicit formulas for all values of the lower triangle of the full tridiagonal inverse matrix $T^{-1}$ \eqref{tridiagser}. E.g., by setting $n=\frac{N}{2}$ in Theorem~\ref{thm:explicit} we obtain the formulas for the first and the second columns of $T^{-1}$ and so on.
\end{remark}
\begin{lemma}
\label{lemtksymmetry}
The matrix $T^{-1}$ satisfies the following symmetries
\begin{align*}
T^{-1}_{j+2i,j} = T^{-1}_{j,j+2i},\ 
T^{-1}_{j+2i+1,j} = -T^{-1}_{j,j+2i+1},\\
T^{-1}_{j - 2i,j} = T^{-1}_{j,j - 2i},\ 
T^{-1}_{j - 2i+1,j} = -T^{-1}_{j,j - 2i+1}.
\end{align*}
for all $1\leq j\leq N$ and $i>0$ such that $j+2i+1 \leq N$, and $1\leq j - 2i+1$.
\end{lemma}
\begin{proof}
Without loss of generality, let $1\leq j\leq N$. It must hold that
\[
T^{-1}_{j-1, j}T_{j, j-1} + T^{-1}_{j,j}T_{j,j} + T^{-1}_{j+1, j}T_{j,j+1} = 1,\qquad T^{-1}_{j, j-1}T_{j-1, j} + T^{-1}_{j,j}T_{j,j} + T^{-1}_{j, j+1}T_{j+1,j} = 1,
\]
as it holds that $T_{j, j-1} = -T_{j-1, j}$ and $T_{j, j+1} = -T_{j+1, j}$ we obtain
\begin{equation}
\label{eqfirstcond}
T^{-1}_{j-1, j} = -T^{-1}_{j, j-1}\text{ \ and \ }T^{-1}_{j+1, j} = - T^{-1}_{j, j+1}.
\end{equation}
From the same argument applied to other column, we get
\[
T^{-1}_{j, j}T_{j+1, j} + T^{-1}_{j+1,j}T_{j+1,j+1} + T^{-1}_{j+2,j}T_{j+1,j+2} = 0,\qquad T^{-1}_{j, j}T_{j, j+1} + T^{-1}_{j,j+1}T_{j+1,j+1} + T^{-1}_{j, j+2}T_{j+2,j+1} = 0.
\]
From \eqref{eqfirstcond} and $T_{j+1,j+2} = -T_{j+2,j+1}$ it follows that $T^{-1}_{j+2,j} = T^{-1}_{j, j+2}$.

Now, it is clearly seen that proceeding further analogously, and we obtain the claim.
\end{proof}
\begin{remark}
By iterative application of Theorem~\ref{thm:explicit} we can obtain the explicit formulas for all values of the lower triangle of the full tridiagonal inverse matrix $T^{-1}$ \eqref{tridiagser}. Then, the symmetries from Lemma~\ref{lemtksymmetry} provide the explicit formulas for all values of the full tridiagonal inverse matrix $T^{-1}$.
\end{remark}
\begin{lemma}
\label{lem:tkinvlowertriang}
It holds that 
\begin{equation}
\label{eq:norms}
\left\|T^{-1}_{\ro{N-2n+1:N},\co{N-2n+c}}\right\|_{\ellmat} = \left\{\begin{array}{l}
\left\|\widetilde{I}_{\co{c}}\right\|_{\ellmat},\text{ for }n=1,\\
\left\|\widetilde{I}_{\co{c}}\right\|_{\ellmat} + \sum_{i=1}^{n-1}\mu_k^{ i } \left(a_{n-i} + \hat{a}_{n-i}\right)\prod_{p = 1}^{i-1}{\hat{a}_{n-p}|\widetilde{I}_{2c}}|\text{, for }1 < n \leq\frac{N}{2},
\end{array}
\right.
\end{equation}
where $c=1$ or $2$, $\widetilde{I}$ is the diagonal block \eqref{eq:I} and $\|\widetilde{I}_{\co{c}}\|_{\ellmat}=|\widetilde{I}_{1c}|+|\widetilde{I}_{2c}|$. In the sum above for the case $i=1$ we put $\prod_{p = 1}^{i-1}{\hat{a}_{n-p}}=1$.
\end{lemma}
\begin{proof}
Consider $\widetilde{T}^{-1}$ -- the bottom right corner sub-matrix of $T^{-1}$ that spans the following diagonal elements of $T$: $d_{N-2n+1}, \dots, d_N$. By Lemma~\ref{lemblocks} $\widetilde{T}^{-1}$ is equal to
\[
\widetilde{T}^{-1} = T^{-1}(d_{N-2n+1} + \mu_k^2b_{\frac{N-2n}{2}},\ d_{N-2n+2},\ \dots,\ d_{N}).
\]
$\widetilde{T}^{-1}$ embeds in $T^{-1}$ as
\[
\widetilde{T}^{-1}=T^{-1}_{\ro{N-2n+1:N},\co{N-2n+1:N}},
\]
it is the $2n\times 2n$ dimensional bottom right corner submatix of $T^{-1}$. Knowing that $\ro{1:2}$ and $\co{1:2}$ of $\widetilde{T}^{-1}$ embeds in $\ro{N-2n+1:N-2n}$ and $\co{N-2n+1:N-2n}$ of $T^{-1}$ respectively, the recurrence relation shown in Theorem~\ref{thm:explicit} takes the form
\[
T^{-1}_{\ro{N-2n+2i+1:N-2n+2i+2},\co{N-2n+1:N-2n+2}} = \mu_k\begin{bmatrix}-a_{n-i}\\\hat{a}_{n-i}\end{bmatrix}T^{-1}_{\ro{N-2n+2i},\co{N-2n+1:N-2n+2}},
\]
for $i=1,\dots,n-1$. After unveiling the recurrence relation we get for $c=1$ or $c=2$ (denoting the odd/even column respectively)
\begin{equation}
\label{eq:terms}
\left\|T^{-1}_{\ro{N-2n+2i+1:N-2n+2i+2},\co{N-2n+c}}\right\|_{\ellmat} = \left\{ \begin{array}{ll}
\mu_k^{ i } \left(a_{n-i} + \hat{a}_{n-i}\right)\prod_{p=1}^{i-1}{\hat{a}_{n-p}|\widetilde{I}_{2c}}|&\text{ for }1 < i < n,\\
\mu_k \left(a_{n-1} + \hat{a}_{n-1}\right)|\widetilde{I}_{2c}|&\text{ for }i=1,\\
\left\|\widetilde{I}_{\co{c}}\right\|_{\ellmat}&\text{ for }i=0,
\end{array}\right.
\end{equation}
Summing up all elements in rows $N-2n+1:N$ we have
\begin{equation}
\label{eq:norms}
\left\|T^{-1}_{\ro{N-2n+1:N},\co{N-2n+c}}\right\|_{\ellmat} = 
\left\|\widetilde{I}_{\co{c}}\right\|_{\ellmat} + \sum_{i=1}^{n-1}\mu_k^{ i } \left(a_{n-i} + \hat{a}_{n-i}\right)\prod_{p = 1}^{i-1}{\hat{a}_{n-p}|\widetilde{I}_{2c}}|,
\end{equation}
where for $i = 1$ we put $\prod_{p = 1}^{i-1}{\hat{a}_{n-p}=1}$ and the first term in the sum \eqref{eq:norms} is equal to $\mu_k \left(a_{n-1} + \hat{a}_{n-1}\right)|\widetilde{I}_{2c}|$ according to \eqref{eq:terms}.
\end{proof}
We will often need a formula covering $\|T_k^{-1}\|$ whole columns, not only the lower triangle as in Lemma~\ref{lem:tkinvlowertriang}. Combining Lemma~\ref{lem:tkinvlowertriang} with Lemma~\ref{lem:tkinvuppertriang2} proven in the appendix we obtain
\begin{lemma}
The following bound holds
\[
\left\|\left(T_k^{-1}\right)_{\co{N-2n+c}}\right\|_{\ellmat} \leq 
\|\widetilde{I}_{\co{c}}\|_{1} + 
\sum_{i=1}^{n-1}\mu_k^{ i } \left(a_{n-i} + \hat{a}_{n-i}\right)\prod_{p = 1}^{i-1}{\hat{a}_{n-p}|\widetilde{I}_{2c}}|+
\sum_{i=1}^{\frac{N}{2}-n}\mu_k^{ i } 
\left(a_{n+i} + \hat{a}_{n+i}\right)\prod_{p = 1}^{i-1}{\hat{a}_{n+p}|\widetilde{I}_{2c}}|
\]
for all $n=1,\dots,\frac{N}{2}$ and $c=1\text{ or }2$, where $\widetilde{I}$ is the diagonal block \eqref{eq:I} and $\|\widetilde{I}_{\co{c}}\|_{\ellmat}=|\widetilde{I}_{1c}|+|\widetilde{I}_{2c}|$. In case $n=1$ we put the first sum is zero, in case $n=\frac{N}{2}$ the second sum is zero.
\end{lemma}
We proceed with proving the bound for $\|T^{-1}\|_1$ essential for the ultimate result of this section.  
\begin{lemma}
\label{lem:tkinvsplit}
It holds that
\[
\left\|T^{-1}\right\|_{\ellmat} < \left\|T_{\ro{1:2\lfloor\mu_k\rfloor}}^{-1}\right\|_{\ellmat} + \frac{1}{2\mu_k}
\]
for all $N>2\mu_k$, the bound is uniform with respect to (even) approximation dimension $N$.
\end{lemma}
\begin{proof}
We start by showing
\begin{equation}
\label{eq:tkinvtail}
\left\| T^{-1}_{\substack{\ro{2\lfloor\mu_k\rfloor+1\colon N}\\\co{l}}}\right\|_{\ellmat} <\frac{1}{2\mu_k}
\end{equation}
for any $1 \leq l \leq 2\mu_k$, where $n$ is such that $N-2n+1 = l$ if $l$ is odd, and $N-2n+2 = l$ if $l$ is even.

We use the equality derived in Lemma~\ref{lem:tkinvlowertriang}, i.e.,
\begin{equation}
\label{eq:tkinvlowertriang2}
\left\|T^{-1}_{\ro{N-2n+1:N},\co{N-2n+c}}\right\|_{\ellmat} = 
\left\|\widetilde{I}_{\co{c}}\right\|_{\ellmat} + \sum_{i=1}^{n-1}\mu_k^{ i } \left(a_{n-i} + \hat{a}_{n-i}\right)\prod_{p = 1}^{i-1}{\hat{a}_{n-p}|\widetilde{I}_{2c}}|,
\end{equation}
where for the case $n=1$ the second term in the sum is not present. Recall $\widetilde{I}$ is the corresponding diagonal block of $T^{-1}$ \eqref{eq:I} i.e.,
\[
\widetilde{I} =  \begin{bmatrix}d_{N-2n+1}+\mu_k^2b_{\frac{N-2n}{2}} & \mu_k \\ -\mu_k & d_{N-2n+2} +\mu_k^2a_{n-1}\end{bmatrix}^{-1},
\]
where $D=d_{N-2n+1}d_{N-2n+2} +\mu_k^2a_{n-1} d_{N-2n+1}+d_{N-2n+2}\mu_k^2b_{(N-2n)/2} +\mu_k^4a_{n-1} b_{(N-2n)/2}+\mu_k^2$ is the determinant. We have the following upper bounds for $|\widetilde{I}_{21}|$, $|\widetilde{I}_{22}|$ following from Lemma~\ref{lema}
\begin{gather}
\label{eq:Iinvboundsinline}
    \left|\widetilde{I}_{21}\right| = \frac{\mu_k}{D}  < \hat{a}_{n} < \frac{1}{\mu_k},\quad\left|\widetilde{I}_{22}\right| = \frac{d_{N-2n+1} +\mu_k^2b_{(N-2n)/2}}{D}< b_{(N-2n+2)/2} < \frac{1}{\mu_k}.
\end{gather}
%
%
Let us take $n$ such that (it is possible as $N$ is even and $N>2\mu_k$)
\begin{equation}
\label{eq:cond}
N - 2n + 2= 2\lfloor\mu_k\rfloor,
\end{equation}
then the assumption of Lemma~\ref{lema2}, i.e. $ n - i \leq \frac{N}{2}-\lfloor\mu_k\rfloor + 1$ is satisfied for all $i=1,\dots,n$ with $p=1$ and hence it holds that
\[
    \hat{a}_{n - i } + a_{n - i } \leq \frac{1}{4\mu_k} + \frac{1}{16\mu_k} = \frac{5}{16\mu_k}\text{ and for $i>1$ } \prod_{p=1}^{i-1}{\hat{a}_{n-p}}\leq \frac{1}{17^{i-1}\mu_k^{i-1}},
\]
we bound all terms appearing in \eqref{eq:norms} as follows
\begin{subequations}
\label{eq:aahatbounds}
\begin{align}
&\mu_k \left(a_{n-1} + \hat{a}_{n-1}\right)|\widetilde{I}_{2c}| < \frac{5}{16\mu_k}< \frac{6}{17\mu_k},\\
&\mu_k^{ i } \left(a_{n-i} + \hat{a}_{n-i}\right) \prod_{p=1}^{i-1}{\hat{a}_{n-p}|\widetilde{I}_{2c}}| < 
\mu_k^i\frac{5}{16\mu_k}\frac{1}{17^{i-1}\mu_k^{i-1}}\frac{1}{\mu_k} < \frac{6}{17^i\mu_k}\text{, for }i>1.
\end{align}
\end{subequations}
and the sum in \eqref{eq:norms} is bounded by
\[
\sum_{i=1}^{n-1}\mu_k^{ i } \left(a_{n-i} + \hat{a}_{n-i}\right)\prod_{p = 1}^{i-1}{\hat{a}_{n-p}|\widetilde{I}_{2c}}|<\frac{6}{\mu_k}\left(\sum_{i=1}^{n-1}{\frac{1}{17^i}}\right) < \frac{6}{16\mu_k}.
\]

Next, to show the bound $\left\| T^{-1}_{\substack{\ro{2\lfloor\mu_k\rfloor+1\colon N}\\\co{l}}}\right\|_{\ellmat} <\frac{1}{2\mu_k}$ for  $l> 2\lfloor\mu_k\rfloor+1$ we proceed as follows. Let $n$ be such that $N-2n+1 = l$ if $l$ is odd, and $N-2n+2=l$ if $l$ is even. As assumptions of Lemma~\ref{lema2} are satisfied ($d_{N-2n+1},\ d_{N-2n+2} \geq 2\mu_k$) with $p=1$, we have that
$\left|\widetilde{I}_{11}\right|=\frac{d_{N-2n+2} +\mu_k^2a_{n-1}}{D} < a_{n} < \frac{1}{4\mu_k},\  \left|\widetilde{I}_{22}\right| = \frac{d_{N-2n+1} +\mu_k^2b_{(N-2n)/2}}{D}< b_{(N-2n+2)/2} < \frac{1}{4\mu_k}$,\ $\left|\widetilde{I}_{12}\right|,\,\left|\widetilde{I}_{21}\right| = \frac{\mu_k}{D} < \frac{1}{17\mu_k}$. 

We have that \eqref{eq:tkinvlowertriang2} is bounded by (using the bounds \eqref{eq:aahatbounds} and the explicit bounds for $\left|\widetilde{I}_{21}\right|$ and $\left|\widetilde{I}_{22}\right|$ above)
\[
\left\|T^{-1}_{\ro{N-2n+1:N},\co{N-2n+c}}\right\|_{\ellmat} \leq
\frac{1}{4\mu_k} + \frac{1}{17\mu_k}+\frac{6}{4\mu_k}\left( \sum_{i=1}^{n-1}{\frac{1}{17^i}}\right) < \frac{1}{\mu_k}\left(\frac{1}{4}+\frac{1}{17}+\frac{3}{32}\right).
\]

Using Lemma~\ref{lem:tkinvuppertriang2} we similarly bound the terms in the upper triangle ($\ro{2\lfloor\mu_k\rfloor+1:N-2n}$). Hence
\begin{equation}
\label{eq:bounduppertriangleterms}
\left\|T^{-1}_{\ro{2\lfloor\mu_k\rfloor+1:N-2n},\co{N-2n+c}}\right\|_{\ellmat}<\frac{6}{4\mu_k}\left( \sum_{i=1}^{\frac{N}{2}-n-\lfloor\mu_k\rfloor+1}{\frac{1}{17^i}}\right).
\end{equation}
Hence, the final bound is
\[
\left\|T^{-1}_{\ro{2\lfloor\mu_k\rfloor+1:N},\co{N-2n+c}}\right\|_{\ellmat} <\frac{1}{\mu_k}\left(\frac{1}{4}+\frac{1}{17}+2\cdot\frac{3}{32}\right) < \frac{1}{2\mu_k}.
\]
\end{proof}

We have the following trivial component-wise uniform bound for $T_k^{-1}$ (derived analogously as Lemma 4.8 from \cite{cyranka_mucha}).
\begin{lemma}
\label{lem:tkinvelement}
Let $N> 2\mu_k$, the following entry-wise bound holds
\[
\left|T_{i,l}^{-1}\right| \leq \frac{\sqrt{2}}{\mu_k}
\]
for all $1\leq i, l\leq N$.
\end{lemma}
\begin{proof}
We will use the upper bounds from Lemma~\ref{lema}, the explicit formula for diagonal blocks from Corollary~\ref{corblocks}, recursive formulas from Theorem~\ref{thm:explicit} and the symmetry from Lemma~\ref{lemtksymmetry}. Let $\widetilde{I}$ be the diagonal block corresponding to $l$-th column of $T_k^{-1}$ \eqref{eq:I}, i.e.,
\[
\widetilde{I} =  \begin{bmatrix}d_{N-2n+1}+\mu_k^2b_{\frac{N-2n}{2}} & \mu_k \\ -\mu_k & d_{N-2n+2} +\mu_k^2a_{n-1}\end{bmatrix}^{-1},
\]
where $l = N-2n+1$ for $l$ odd, and $l=N-2n+2$ for $l$ even, 
and $D=d_{N-2n+1}d_{N-2n+2} +\mu_k^2a_{n-1} d_{N-2n+1}+d_{N-2n+2}\mu_k^2b_{(N-2n)/2} +\mu_k^4a_{n-1} b_{(N-2n)/2}+\mu_k^2$ is the determinant of $\widetilde{I}$.

The following entry-wise bounds hold for $\widetilde{I}$:
\begin{subequations}
\label{eq:eqIbds}
\begin{align}
    \left|\widetilde{I}_{11}\right| &= \frac{d_{N-2n+2} +\mu_k^2a_{n-1}}{D}< a_{n} < \frac{\sqrt{2}}{\mu_k},\\
    \left|\widetilde{I}_{12}\right|, \left|\widetilde{I}_{21}\right| &= \frac{\mu_k}{D}  < \hat{a}_{n} < \frac{1}{\mu_k},\label{eqI12bd}\\
    \left|\widetilde{I}_{22}\right| &= \frac{d_{N-2n+1} +\mu_k^2b_{(N-2n)/2}}{D}< b_{(N-2n+2)/2} < \frac{1}{\mu_k},\label{eqI22bd}
\end{align}
\end{subequations}

To bound the elements below the diagonal block we use the recursive formulas from Theorem~\ref{thm:explicit}, i.e. the first block slot below $\widetilde{I}$ is given by 
$\mu_k\begin{bmatrix}-a_{n-1}\\\hat{a}_{n-1}\end{bmatrix}\widetilde{I}_{\ro{2}}$ where $a_{n-1}$, $\hat{a}_{n-1}$ are upper bounded by $\frac{\sqrt{2}}{\mu_k}$, each component of $\widetilde{I}_{\ro{2}}$ is upper bounded by $\frac{1}{\mu_k}$ (using Lemma~\ref{lema}). Applying the recursion from Theorem~\ref{thm:explicit} repetitively for all of the diagonal blocks we obtain that the uniform upper bound $\frac{\sqrt{2}}{\mu_k}$ is true for the lower  triangle of $T_k^{-1}$ (including the diagonal blocks). From the symmetry formula Lemma~\ref{lemtksymmetry} this uniform upper bound is true for all entries in $T_k^{-1}$.
\end{proof}

\begin{theorem}
\label{thm:trivialTkinvbd}
It holds that
\[
\left\|T^{-1}\right\|_{\ellmat} \leq 2\sqrt{2} + \frac{1}{2\mu_k}.
\]
for all $N>2\mu_k$ (the bound is uniform with respect to even approximation dimension $N$).
\end{theorem}
\begin{proof}
Using Lemma~\ref{lem:tkinvsplit} we have that $\left\|T^{-1}\right\|_{\ellmat} \leq \left\|T_{\ro{1:2\lfloor\mu_k\rfloor}}^{-1}\right\|_{\ellmat} + \frac{1}{2\mu_k}$. Using the element-wise bound from Lemma~\ref{lem:tkinvelement} the first term is upper bounded by
\[
\left\|T_{\ro{1:2\lfloor\mu_k\rfloor}}^{-1}\right\|_{\ellmat} \leq 2\lfloor\mu_k\rfloor \frac{\sqrt{2}}{\mu_k} \leq 2\sqrt{2}.
\]
\end{proof}
Using the tighter integral bounds derived in the Appendix, we obtain the following result.

\begin{theorem}
\label{thm:tkimprov}
Let $\mu_k>10$. Then, for all $N>2\mu_k$ (the bounds are uniform with respect to even approximation dimension $N$), it holds that
\[
\left\|T_k^{-1}\right\|_{\ellmat} \le 
\begin{cases}
0.806 + \frac{1}{20}, & \text{if } \mu_k>10,
\\
0.736 + \frac{1}{200}, & \text{if } \mu_k>100,
\\
0.727 + \frac{1}{2000}, & \text{if } \mu_k>10.
\end{cases}
\]


\end{theorem}
\begin{proof}[Proof of Theorem~\ref{thm:tkimprov}]
We present the details in Appendix~\ref{appxspacereg}. 

The fundamental idea behind the proof is to apply in practice Theorem~\ref{thm:explicit} and Lemma~\ref{lemtksymmetry}, which provide explicit recursive formulas for all $T_k^{-1}$ entries. By summing up and unveiling the recursion in the formulas, an explicit formula for the $\ell^1$ norm of the lower triangle (including the diagonal) of $T_k^{-1}$ follows. By analogous computations, we obtain an explicit formula for the $\ell^1$ norm of the upper triangle by reversing the order of the diagonal elements.

Using the derived explicit formulas, $\|T_k^{-1}\|_{\ellmat}$ is bounded by the sequence sum of $a_j$'s and $\hat{a}_j$'s from Def.~\ref{defaahat}. The resulting finite sum is in turn bounded by a definite integral that can be computed explicitly (we performed a symbolic Mathematica computation), and this determines the bound for the first term $\left\|T_{\ro{1:2\lfloor\mu_k\rfloor}}^{-1}\right\|_{\ellmat}$ appearing in Lemma~\ref{lem:tkinvsplit} (for sufficiently large $\mu_k$ as 	in the statement of the Theorem), the remainder according to Lemma~\ref{lem:tkinvsplit} is bounded by $1/2\mu_k$.
\end{proof}
\begin{figure}[h!]
\includegraphics[width=\textwidth]{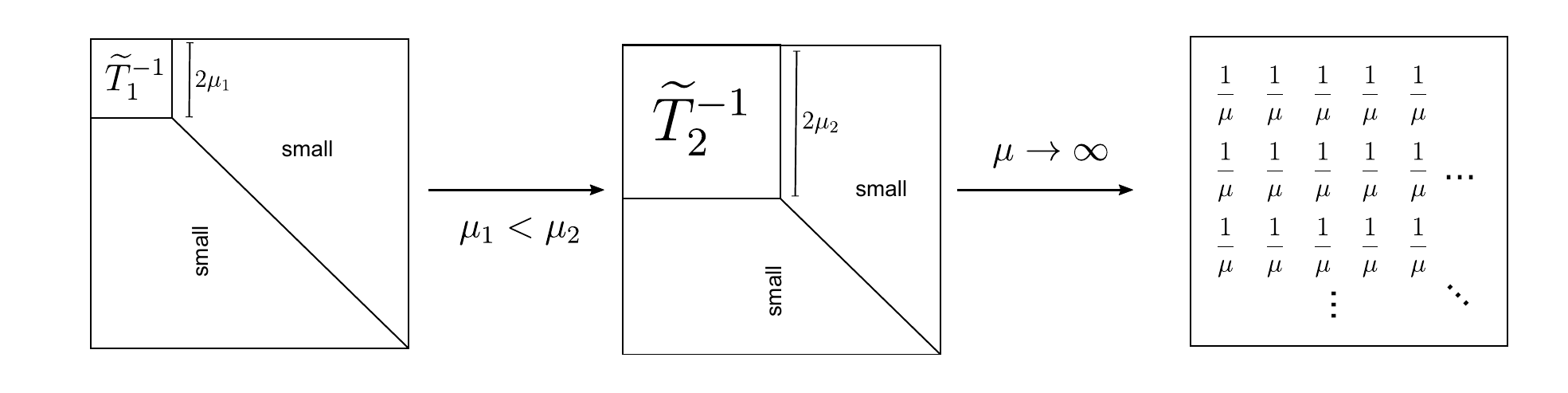}
\caption{ Diagram motivating requirement of performing our analysis to compute the stability of the inverse tridiagonal operator $T_k^{-1}$ with respect to the approximation dimension. 
Notation $\widetilde{T}_1,\widetilde{T}_2$ mean the upper left finite dimensional block of $T_k^{-1}$ having $\mu_1, \mu_2$ as the off-diagonal elements, which are denoted by 
$(T_k^{-1})_{ \ro{1:2\mu},\ \co{1:2\mu} }$. Intuitively, all the entries of the inverse matrices outside $\widetilde{T}_1$ and $\widetilde{T}_2$  are small, as they lie in the regime where the corresponding part of $T_k$ is diagonally dominant.
}
\label{figblock}
\end{figure}

\subsubsection{Uniform bounds for \boldmath$\|\cLpkinv\|_{B(\ell^1)}$\unboldmath\ and the ultimate bound for \boldmath$\|\cL_k^{-1}\|_{B(\ell^1)}$\unboldmath\  }

We are going to use the following well-known result in linear algebra about the inverse of the sum of
an invertible matrix $A$ and a rank one perturbation $uv^T$ ($u,v$ are vectors), known as the
\emph{Sherman-Morrison formula} \cite{SM1, SM2, SM3, SM4, SM5}.

\begin{lemma}[\cite{SM5}]
\label{lemsm}
Suppose $A$ is an \emph{invertible real square matrix} and $u,v$ are \emph{column vectors}. Suppose furthermore that $1 + v^TA^{-1}u\neq 0$. Then 
the \emph{Sherman-Morrison} formula states that 
\[
(A+uv^T)^{-1} = A^{-1} - \frac{A^{-1}uv^TA^{-1}}{1+v^TA^{-1}u}.
\]
Here, $uv^T$ is the outer product of two vectors $u$ and $v$.
\end{lemma}


Now let us apply Lemma~\ref{lemsm} for the inverse of the linear operator that we consider. We recall its form below. $\cLp_{k}$ is being decomposed into a sum of a tridiagonal operator and a single non-zero row operator
\begin{equation}
\label{plusrone2}
  \cLp_{k} = M_k +  U  = \left[ \begin{array}{cccc} 1&0&\cdots&0\\\eigk&\ &\ &\ \\0&\ &T_k&\ \\\vdots &\ &\ &\ \end{array}\right] + 
\left[ \begin{array}{cccc}0&-2&2&\dots\\\ &\ &\ &\ \\\ &\ &0&\ \\\ &\ &\ &\ \\ \end{array} \right],
\end{equation}
where 
\begin{equation}
\label{tridiag2}
T_{k} = \left[ \begin{array}{ccccc} 
 2 &-\eigk&0&\cdots&\ \\
\eigk& 4 &-\eigk&0&\cdots\\
\ddots&\ddots&\ddots&\ddots&\ddots\\
\dots&0&\eigk& 2(N-1) &-\eigk\\
\ &\dots&0&\eigk& 2N
\end{array}\right].
\end{equation}

Observe that in our case, we apply the Sherman-Morrison formula for \eqref{plusrone2}, and we put
\begin{equation*}
A+uv^T = \cLpk,\ A = M_k,\ uv^T = U,\ u = \left[1,0,\dots,0\right]^T,\ 
v^T = 2\left[0,-1,1,\dots,(-1)^{k+1},\dots\right].
\end{equation*}

$M_k$ is invertible; being a tridiagonal matrix, it can be checked by the recursive determinant computation that it is 
bounded away from zero for all $k$ and $N$.

\begin{lemma}
\label{akinv}
The inverse of $M_k$ in \eqref{plusrone} is given by
\begin{equation}
\label{eqakinv}
M_k^{-1} = 
\begin{bmatrix} 1 & 0\\\begin{smallmatrix}\eigk\\0\\\vphantom{\int\limits^x}\smash{\vdots}\end{smallmatrix}&T_k\end{bmatrix}^{-1} = 
\left[ \begin{array}{cccc} 1&0&\cdots&0\\-\eigk\cdot \left(T_{k}^{-1}\right)_{\co{1}}&\ &T_k^{-1}&\ \end{array}\right].
\end{equation}
\end{lemma}

\begin{proof} Let us denote $\inv \bydef M_k^{-1}$.
From solving the linear system of equations
\[
\begin{bmatrix} 1 & 0\\\begin{smallmatrix}\eigk\\0\\\vphantom{\int\limits^x}\smash{\vdots}\end{smallmatrix}&T_k\end{bmatrix}\cdot
\left[\begin{array}{cc}\inv_{11}&\inv_{12}\\\inv_{21}&\inv_{22}\end{array}\right]=
\left[\begin{array}{cc}1&0\\0&\id \end{array}\right].
\]
we have that 
\begin{equation*}
\begin{array}{ll}
\inv_{11}=1,&
\inv_{12}=0,\\
\inv_{21}=T_k^{-1}\begin{bmatrix}-\eigk\\0\\\vdots\\0\end{bmatrix} = -\eigk\cdot \left(T_{k}^{-1}\right)_{\co{1}},&
\inv_{22}=T_k^{-1}.
\end{array}
\end{equation*}
\end{proof}
\begin{corollary}
\label{lkcor}
Let $M_k$ be the first term in the decomposition of $\cLp_{k}$ in \eqref{plusrone2}. 
\begin{equation}
\label{akuk}
\cLpkinv = M_k^{-1} - \frac{M_k^{-1}U M_k^{-1}}{1 + v^TM_k^{-1}u}.
\end{equation}
\end{corollary}
\begin{theorem}
\label{thmmain2}
Let $\mu_k > 0$. Let $T_k$ be the tridiagonal matrix \eqref{tridiag2}, and $\cLp_k$ be the matrix having the diagonal block structure as in \eqref{plusrone2}.

For all $k$ such that $\mu_k \geq 0$ and for all $N>0$ it holds that
\[
\|\cLpkinv\|_{B(\ell^1)}\leq 
\max\left\{2\left\| T_k^{-1}\right\|_{\ellmat}, 1\right\},
\]
where $1$ above comes from the fact that $\|\cLpkinv_{\co{1}}\|_{\ell^1}=1$.
\end{theorem}
\begin{proof}
We denote 
\begin{equation}
\label{eq:akinv}
[m_{ij}]_{i,j=1}^{N+1} = M_k^{-1} = \left[ \begin{array}{cccc} 1&0&\cdots&0\\-\eigk \left(T_{k}^{-1}\right)_{\co{1}}&\ & T_k^{-1}&\ \end{array}\right].
\end{equation}
We analyze the numerator and denominator of the perturbation term in the explicit form for $\cLpkinv$ \eqref{akuk}. Observe that
\[
M_k^{-1}UM_k^{-1} = M_k^{-1}uv^TM_k^{-1} = 2\begin{bmatrix}m_{11}\\m_{21}\\\vdots\\m_{N+1,1}\end{bmatrix}\cdot\begin{bmatrix} \sum_{j=2}^{N+1}{(-1)^{j+1}m_{j,1}}&\sum_{j=2}^{N+1}{(-1)^{j+1}m_{j,2}}&\dots&\sum_{j=2}^{N+1}{(-1)^{j+1}m_{j,N+1}}\end{bmatrix}
\]
and
\[
1 + v^T M_k^{-1}u = 1 + 2\sum_{j=2}^{N+1}(-1)^{j+1} m_{j,1}.
\]
We will bound $\left\|[\cLpkinv]_{\co{l}}\right\|_{\ell^1}$ by considering two cases: $l=1$, and $l\geq 2$. 

\paragraph{Case $l=1$} it is easy to see a lot of terms in numerator get canceled
\begin{equation}
\label{lbound}
[\cLpkinv]_{\co{1}} = \left( m_{i,1}-2\frac{m_{i,1}\sum_{j=2}^{N+1}{(-1)^{j+1}m_{j,1}}}{1+2\sum_{j=2}^{N+1}{(-1)^{j+1}m_{j,1}}} \right)^{N+1}_{i=1}=
\left( \frac{m_{i,1}}{1+2\sum_{j=2}^{N+1}{(-1)^{j+1}m_{j,1}}} \right)^{N+1}_{i=1},
\end{equation}
where  $m_{1,1}=1$, and for $j>1\colon\,m_{j,1} = -\eigk \left(T_{k}^{-1}\right)_{{j-1},1}$.

Knowing the explicit formulas of elements in $(T_k^{-1})_{\co{1}}$ with precise sign information from Lemma~\ref{lemtksymmetry}, we obtain that the alternating sum from the denominator is in fact the norm
\begin{equation}
\label{eqpossum}
\sum_{j=2}^{N+1}(-1)^{j+1}m_{j1} = \sum_{j=1}^{N}{\left|\mu_k(T_k^{-1})_{j,1}\right|} = |\mu_k|\|(T_k^{-1})_{\co{1}}\|_{\ellmat}.
\end{equation}
Hence it holds that ($m_{1,1}=1$)
\[
\left\| [\cLpkinv]_{\co{1}}\right\|_{\ell^1} = \frac{1 + 2|\eigk|\|(T_k^{-1})_{\co{1}}\|_{\ellmat}}{1 + 2|\eigk|\|(T_k^{-1})_{\co{1}}\|_{\ellmat}} = 1.
\]

\paragraph{Case $l\geq 2$}
\begin{equation}
\label{eqlkinvl3}
[\cLpkinv]_{\co{l}} = \left( m_{i,l}-2\frac{ m_{i,1}\sum_{j=2}^{N+1}{(-1)^{j+1}m_{j,l}}}{1+2\sum_{j=2}^{N+1}{(-1)^{j+1}m_{j,1}}} \right)_{i=1}^{N+1}.
\end{equation}
Taking the absolute value of each entry of the vector \eqref{eqlkinvl3}, applying the triangle inequality, and summing up we obtain the following estimate for all $l\geq 2$
\begin{align}
\label{eq:lkinv_form}
\left\| [\cLpkinv]_{\co{l}} \right\|_{\ell^1} &\leq \frac12|m_{1,l}|+\sum_{j=2}^{N+1}{|m_{j,l}}| + \frac{\left(|m_{1,1}|+2\sum_{j=2}^{N+1}{|m_{j,1}|}\right)\sum_{j=2}^{N+1}{|m_{j,l}|}}{1+2\sum_{j=2}^{N+1}{|m_{j,1}|}}\\
& \leq \sum_{i=1}^{N+1}{|m_{i,l}|} + \frac{\left(1 + 2\sum_{j=2}^{N+1}{|m_{j,1}|}\right)\sum_{j=2}^{N+1}{|m_{j,l}|}}{1+2\sum_{j=2}^{N+1}{|m_{j,1}|}} \nonumber\\
& 
\leq 2\sum_{i=1}^{N+1}{|m_{i,l}|} \leq 2\sum_{i=1}^{N}{|(T_k^{-1})_{i,l-1}|} \leq
2\left\|T_k^{-1}\right\|_{\ellmat}. 
\qedhere
\end{align}
\end{proof}

We have the following straightforward corollary summarizing the results derived in this section.
\begin{corollary} \label{cor:LkinvProj}
Let $T_k$ be the tridiagonal matrix, and $\cLp_k$ be the matrix having the diagonal block structure as in \eqref{plusrone2}, $\|\cdot\|_{B(\ell^1)}$ is defined in \eqref{eq:Bell1matrix}. Then, for all $N>2\mu_k$ (the bounds are uniform with respect to even approximation dimension $N$), 
\begin{equation} \label{eq:upper_uniform_bound_large_k_LNinv}
\|\cLpkinv\|_{B(\ell^1)} \le
C(\mu_k) \bydef
\begin{cases}
1.612 + 0.1, & \text{if } \mu_k \in (10,100]
\\
1.472 + 0.01, & \text{if } \mu_k \in (100,1000]
\\
1.454 + 0.001, & \text{if } \mu_k > 1000.
\end{cases}
\end{equation}
\end{corollary}

We proceed in bounding the full infinite dimensional inverse linear operator $\cL_k^{-1}$, our main tool is summarized in the following theorem.
\begin{theorem}
\label{thm:cauchy}
Let $\mu_k>10$. Let $T_k$ be the tridiagonal operator, and $\cL_k$ be the (infinite dimensional) operator. Let $A_k$ be the approximate inverse (infinite dimensional), where its block dimensions depend on $N$, defined as 
\[
A_k \bydef \left[\def\arraystretch{1.25}\begin{array}{c|c}
\cLpkinv & -[\cLpkinv]_{\co{1}}v\Omega^{-1}\\ \hline
0 & \Omega^{-1}
\end{array}\right].
\]

For any $\varepsilon>0$ and $\mu_k > 10$ there exists $N(\varepsilon,\mu_k)$, such that 
\[
\|I-A_k\cL_k\|_{B(\ell^1)} \leq \varepsilon,\text{ for all }N>N(\varepsilon,\mu_k)\text{ and }N\text{ even,}
\]
consequently
\[
\|\cL_{k}^{-1}\|_{B(\ell^1)} \leq \frac{C(\mu_k)}{1-\varepsilon}\text{, and }A_k\to \cL_k^{-1}\text{ as }\frac{N}{2}\to\infty,
\]
where $C(\mu_k)$ is the constant from Corollary~\ref{cor:LkinvProj} given in \eqref{eq:upper_uniform_bound_large_k_LNinv}.
\end{theorem}
\begin{proof}
We apply Lemma~\ref{lem:bounding_invLk_small_k}. Our goal is to show that $\rho_k<1$, then from Lemma~\ref{lem:bounding_invLk_small_k} $\cL_k$ is a boundedly invertible operator on $\ell^1$ with
\[
\| \cL_k^{-1} \|_{B(\ell^1)} \le \frac{\beta_k}{1-\rho_k},
\]
where $\rho_k \bydef \frac{\mu_k}{2} \max \{ \rho^{(1)}, \rho^{(2)}, \rho^{(3)} \}$, where $\rho^{(1)}, \rho^{(2)}, \rho^{(3)}$ are given by \eqref{eq:rho1}, \eqref{eq:rho2},  \eqref{eq:rho3}, and $\beta_k=\|\cLpkinv\|_{B(\ell^1)}$. In Theorem~\ref{thmmain2} we show that $\|[\cLpkinv]_{\co{1}}\|_{\ell^1}$ is equal to $1$ for all $N>0$ and $\mu_k>0$. Hence it holds that 
\[
    \rho^{(1)} \le \frac{2}{N+1},\quad \rho^{(3)} \le \frac{3}{N+1}.
\]
It reminds us to estimate $\rho^{(2)}$. We proceed using the triangle inequality. 
\[
\rho^{(2)}\leq \|[\cLpkinv]_{\co{N+1}}\|_{\ell^1} + \frac{1}{N+2}\|[\cLpkinv]_{\co{1}}\|_{\ell^1} + \frac{1}{N+2}
\le \|[\cLpkinv]_{\co{N+1}}\|_{\ell^1} + \frac{2}{N+2}.
\]

Let us fix $0<\varepsilon < 1$. Obviously $\frac{\mu_k}{2}\rho^{(1)}, \frac{\mu_k}{2}\rho^{(3)}$ are smaller than $\varepsilon$ for $N>N(\varepsilon,\mu_k)$. Where $N(\varepsilon,\mu_k)=2^p\mu_k+1$ with $p$ sufficiently large to be adjusted later-on. 

We now show that $\frac{\mu_k}{2}\rho^{(2)}\leq\frac{\mu_k}{2}\|[\cLpkinv]_{\co{N+1}}\|_{\ell^1} + \frac{\mu_k}{N} < \varepsilon$ for all $N$ s.t. $N> N(\varepsilon,\mu_k)$ and $N$ even.

From \eqref{eq:lkinv_form} in the proof of Theorem~\ref{thmmain2} it follows that $\|[\cLpkinv]_{\co{N+1}}\|_{\ell^1}\leq 2\|(T_k^{-1})_{\co{N}}\|_1$. The following upper bounds hold for the right-bottom corner diagonal block, corresponding to the last column of $T_k^{-1}$ (compare \eqref{eq:eqIbds})
\begin{align}
    \left|\widetilde{I}_{11}\right| &= \frac{d_{N} }{D}< a_{1},\\
    \left|\widetilde{I}_{12}\right|, \left|\widetilde{I}_{21}\right| &= \frac{\mu_k}{D}  < \hat{a}_{1},\label{eq:i12Nbd}\\
    \left|\widetilde{I}_{22}\right| &= \frac{d_{N-1} +\mu_k^2b_{(N-2)/2}}{D}< b_{N/2},\label{eq:i22Nbd}
\end{align}
where $D$ is the determinant of $\widetilde{I}$. From Lemma~\ref{lema2} it follows that for $N\geq N(\varepsilon,\mu_k)=2^p\mu_k+1$ and $N$ even it holds that (setting $j=1$)
\begin{align}
    b_{N/2} &\leq \frac{1}{2^{p+1}\mu_k},\label{eq:bN2}\\
    \hat{a}_1 &\leq \frac{1}{(2^{2p+2} +1)\mu_k},\label{eq:hata1}
\end{align}
hence 
\begin{equation}
\label{aaa:sum1}
\|\widetilde{I}_{\co{1}}\|_{1}=|\widetilde{I}_{12}|+|\widetilde{I}_{12}|<\hat{a}_1+b_{N/2}<\frac{1}{2^{p}\mu_k}.
\end{equation}

The sum of absolute values of $T_k^{-1}$ upper triangle from Lemma~\ref{lem:tkinvuppertriang2} is bounded by (setting $n=1$)
\[
\left\|\left(T^{-1}_k\right)_{\ro{1:N-2},\co{N}}\right\|_{\ellmat} \leq  \sum_{i=1}^{\frac{N}{2}-1}\mu_k^{ i } 
\left(a_{1+i} + \hat{a}_{1+i}\right)\prod_{p = 1}^{i-1}{\hat{a}_{1+p}|\widetilde{I}_{22}|}.
\]
We split the bound using $\left\|\left(T^{-1}_k\right)_{\ro{1:N-2},\co{N}}\right\|_{\ellmat}\leq \left\|\left(T^{-1}_k\right)_{\ro{1:2\lfloor\mu_k\rfloor},\co{N}}\right\|_{\ellmat}+
\left\|\left(T^{-1}_k\right)_{\ro{2\lfloor\mu_k\rfloor+1:N-2},\co{N}}\right\|_{\ellmat}$. For the first term we have
\begin{equation}
\label{aaa:sum2}
\left\|\left(T^{-1}_k\right)_{\ro{1:2\lfloor\mu_k\rfloor},\co{N}}\right\|_{\ellmat} \leq  \sum_{i=\frac{N}{2}-\lfloor\mu_k\rfloor}^{\frac{N}{2}-1}\mu_k^{ i }
\left(a_{1+i} + \hat{a}_{1+i}\right)\prod_{p = 1}^{i-1}{\hat{a}_{1+p}|\widetilde{I}_{22}|} < \frac{\sqrt{2}+1}{2^{p+1}}
,
\end{equation}
using the upper bounds from Lemma~\ref{lema}, \eqref{eq:i22Nbd}, \eqref{eq:bN2}. And as for the second term, it holds that (analogously to \eqref{eq:bounduppertriangleterms})
\begin{equation}
\label{aaa:sum3}
\left\|\left(T^{-1}_k\right)_{\ro{2\lfloor\mu_k\rfloor+1:N-2},\co{N}}\right\|_{\ellmat} \leq  \sum_{i=1}^{\frac{N}{2}-\lfloor\mu_k\rfloor}\mu_k^{ i } 
\left(a_{1+i} + \hat{a}_{1+i}\right)\prod_{p = 1}^{i-1}{\hat{a}_{1+p}|\widetilde{I}_{22}|} < \frac{6}{2^{p+1}\mu_k}\left( \sum_{i=1}^{\frac{N}{2}-\lfloor\mu_k\rfloor}{\frac{1}{17^i}}\right) < \frac{3}{2^{p+4}\mu_k}.
\end{equation}
Finally summing up \eqref{aaa:sum1}, \eqref{aaa:sum2}, and \eqref{aaa:sum3}
\[
2\|(T_k^{-1})_{\co{N}}\|_1 < \frac{2}{2^{p}\mu_k}+ \frac{2\sqrt{2}+2}{2^{p+1}} + \frac{6}{2^{p+4}\mu_k} < \frac{1}{2^{p}}\left(\frac{2}{\mu_k}+\sqrt{2}+1+\frac{3}{8\mu_k}\right).
\]
Now it reminds to pick $p$ in $N(\varepsilon,\mu_k)=2^p\mu_k+1$, such that $\frac{\mu_k}{2^{p+1}}(\frac{2}{\mu_k}+\sqrt{2}+1+\frac{3}{8\mu_k}) + \frac{1}{2^{p}}< \varepsilon$ and $\frac{1}{2^{p+1}} < \varepsilon$.
\end{proof}
We finally obtain the ultimate bound for $\|\cL_k^{-1}\|_{B(\ell^1)}$ by passing to the limit with $\varepsilon\to 0$ in Theorem~\ref{thm:cauchy}.

\begin{corollary}
\label{cor:Lkinv}
Let $\cL_k$ be the linear operator \eqref{eq:linear_operators} and recall the definition of $C(\mu_k)$ in \eqref{eq:upper_uniform_bound_large_k_LNinv}. Then, 
\[
\|\cL_k^{-1}\|_{B(\ell^1)}
\le C(\mu_k) =
\begin{cases}
1.612 + 0.1,& \text{ for } \mu_k > 10
\\
1.472 + 0.01,& \text{ for } \mu_k > 100
\\
1.454 + 0.001,& \text{ for } \mu_k >1000.
\end{cases}
\]
\end{corollary}

%% file: includes/appendix.tex
\newcommand{\C}{\widetilde{C}}
\newcommand{\mutt}{\mu^{1/2}}
\section{Proof of  Theorem~\ref{thm:tkimprov}}
\label{appxspacereg}

\paragraph{Assume} $\mu_k\geq 1$.

Lemma~\ref{lem:tkinvlowertriang} provides the following explicit formula for $T_k^{-1}$ lower triangle $\ell^1$ norm (including the diagonal)
\begin{equation}
\label{eq:tkinvlowertriangsum}
\left\|\left(T^{-1}_k\right)_{\ro{N-2n+1:N},\co{N-2n+c}}\right\|_{\ellmat} = \left\{\begin{array}{l}
\left\|\widetilde{I}_{\co{c}}\right\|_{\ellmat},\text{ for }n=1,\\
\left\|\widetilde{I}_{\co{c}}\right\|_{\ellmat} + \sum_{i=1}^{n-1}\mu_k^{ i } \left(a_{n-i} + \hat{a}_{n-i}\right)\prod_{p = 1}^{i-1}{\hat{a}_{n-p}|\widetilde{I}_{2c}}|\text{, for }1 < n \leq\frac{N}{2},
\end{array}
\right.
\end{equation}
where $n\in\{1,\dots,\frac{N}{2}\}$ denotes the considered column pair of $T_k^{-1}$ (the largest index $n=\lfloor\mu\rfloor$ denotes the first column pair, and $n=2$ denotes before the last pair), $c=1, 2$ denotes the first and the second element of the pair respectively, and $\widetilde{I}$ is the $2\times 2$ diagonal block \eqref{eq:I} that corresponds to the considered columns.

In order to bound $\ell^1$ norm of full $T_k^{-1}$ matrix, it reminds to bound $\ell^1$ norm of $T_k^{-1}$ upper triangle. We obtain an analogous formula to the one derived in Lemma~\ref{lem:tkinvlowertriang}, but just considering the 'reversed' $T_k$ matrix, which we denote by $\overline{T}_{k}$, i.e.,
\begin{equation}
\label{eq:Tbar}
\overline{T}_{k}\bydef \left[ \begin{array}{cccc} 
 2N & \eigk & 0 & \cdots\\
-\eigk & 2N-2 & \eigk & 0\\
\ddots&\ddots&\ddots&\ddots\\
\dots& 0 & -\eigk & 2
\end{array}\right]
\end{equation}

We define recursive sequences $\{\overline{a}_k\}, \{\hat{\overline{a}}_k\}$ (analogous to $\{a_k\}, \{\hat{a}_k\}$ from Def.~\ref{defaahat})
\begin{definition}
\label{defaabar}
Let $\{\overline{d}_j\}_{j=1}^N$ be the sequence of $\overline{T}_k$ diagonal elements, i.e., $\overline{d}_j=2N-2j+2$, let us define the following recursive sequences ($\overline{a}_0,\hat{\overline{a}}_0 = 0$)
\begin{equation*}
\arraycolsep=2pt\def\arraystretch{1.4}
\begin{array}{lll}
\overline{a}_j &\bydef \frac{ \overline{d}_{N-2j+2} + \overline{a}_{j-1}\mu_k^2}{\overline{d}_{N-2j+1}\overline{d}_{N-2j+2} + \overline{a}_{j-1}\overline{d}_{N-2j+1}\mu_k^2 + \mu_k^2} = \frac{ d_{2j-1} + \overline{a}_{j-1}\mu_k^2}{ d_{2j}d_{2j-1} + \overline{a}_{j-1}d_{2j}\mu_k^2 + \mu_k^2}\\
\hat{\overline{a}}_j &\bydef \frac{ \mu_k }{ \overline{d}_{N-2j+1}\overline{d}_{N-2j+2} + \overline{a}_{j-1}\overline{d}_{2j}\mu_k^2 + \mu_k^2} = \frac{ \mu_k }{ d_{2j}d_{2j-1} + \overline{a}_{j-1}d_{2j}\mu_k^2 + \mu_k^2}
\end{array}
\end{equation*}
for $j=1,2,3,\dots,\frac{N}{2}$.
\end{definition}
And obtain the analogous result to Lemma~\ref{lem:tkinvlowertriang}, but considering the 'reversed' $T_k$ matrix, i.e., $\overline{T}_k$. By 'reversed' we mean that it holds $\left(\overline{T}^{-1}_k\right)_{\ro{2n+1:N},\co{2n-2+c}} = \left(T^{-1}_k\right)_{\ro{1:N-2n},\co{N-2n+c'}}$.
\begin{lemma}
\label{lem:tkinvuppertriang}
It holds that 
\begin{multline}
\label{eq:normsupper}
\left\|\left(\overline{T}^{-1}_k\right)_{\ro{2n+1:N},\co{2n-2+c}}\right\|_{\ellmat} = \left\|\left(T^{-1}_k\right)_{\ro{1:N-2n},\co{N-2n+c'}}\right\|_{\ellmat} =\\\sum_{i=1}^{\frac{N}{2}-n}\mu_k^{ i } 
\left(\overline{a}_{\frac{N}{2}-n+1-i} + \hat{\overline{a}}_{\frac{N}{2}-n+1-i}\right)\prod_{p = 1}^{i-1}{\hat{\overline{a}}_{\frac{N}{2}-n+1-p}|\widetilde{I}_{2c}}|\text{, for }1 \leq n <\frac{N}{2},
\end{multline}
where $c=1, c'=2$ or $c=2, c'=1$. In the sum above for the case $i=1$ we put $\prod_{p = 1}^{i-1}{\hat{\overline{a}}_{\frac{N}{2}-n+1-p}}=1$.
\end{lemma}
\begin{proof}
The lower triangle of $\overline{T}_k^{-1}$ corresponds to the upper triangle $T_k^{-1}$. The same calculations as in the proof of Lemma~\ref{lem:tkinvlowertriang}, but performed for $\overline{T}_{k}$.
\end{proof}

\begin{lemma}
\label{lem:abar}
It holds that 
\[
\overline{a}_j < a_{\frac{N}{2}-j+1}. 
\]
\end{lemma}
\begin{proof}
Follows directly from the definition of $\overline{a}_j$'s Definition~\ref{defaabar}, compare with the definition of $a_j$'s Definition~\ref{defaahat}.
\end{proof}


Using the bound derived in Lemma~\ref{lem:abar} we obtain an upper bound for $\ell^1$ norm of $T_k^{-1}$ upper triangle
\begin{lemma}
\label{lem:tkinvuppertriang2}
It holds that 
\begin{equation}
\label{eq:tkinvuppertriangsum}
\left\|\left(T^{-1}_k\right)_{\ro{1:N-2n},\co{N-2n+c}}\right\|_{\ellmat} <  \sum_{i=1}^{\frac{N}{2}-n}\mu_k^{ i } 
\left(a_{n+i} + \hat{a}_{n+i}\right)\prod_{p = 1}^{i-1}{\hat{a}_{n+p}|\widetilde{I}_{2c}}|\text{, for }1 \leq n <\frac{N}{2},
\end{equation}
where $c=1$ or $2$. In the sum above for the case $i=1$ we put $\prod_{p = 1}^{i-1}{\hat{a}_{n+p}}=1$.
\end{lemma}
\begin{proof}
Put $j=\frac{N}{2}-n+1-i$ in Lemma~\ref{lem:abar}, then $\hat{a}_{\frac{N}{2}-n+1-i} < a_{n+i}$.
\end{proof}

\begin{lemma}
\label{lem:totalbd}
The following bound holds
\begin{equation}
\label{eq:totalbd}
\left\|\left(T_k^{-1}\right)_{\ro{1:N},\co{N-2n+c}}\right\|_{\ellmat} \leq 
\left\|\widetilde{I}_{\co{c}}\right\|_{\ellmat} + 
\mu_k|\widetilde{I}_{2c}|\left(\sum_{\substack{j=1\\j\neq n}}^{N/2}{a_j + \hat{a}_j}\right)<
\sum_{j=1}^{N/2}{a_j + \hat{a}_j},
\end{equation}
for all $n=1,\dots,\frac{N}{2}-1$ and $c=1,2$.
\end{lemma}
\begin{proof}
Applying Lemmas~\ref{lem:tkinvlowertriang},~\ref{lem:tkinvuppertriang2}, the upper bounds from Lemma~\ref{lema} we bound the sums in \eqref{eq:tkinvlowertriangsum} and \eqref{eq:tkinvuppertriangsum} by (for both of $c$ values) $\sum_{j=1}^{n-1}{a_j+\hat{a}_j}$ and $\sum_{j=n+1}^{\frac{N}{2}}{a_j+\hat{a}_j}$ respectively. Using the upper bounds \eqref{eq:eqIbds} we overestimate $\left\|\widetilde{I}_{\co{1}}\right\|_{\ellmat} < a_n+\hat{a}_n$, $\left\|\widetilde{I}_{\co{2}}\right\|_{\ellmat} < b_{(N-2n+2)/2}+\hat{a}_n < a_n+\hat{a}_n$. Hence the final overestimate follows.
\end{proof}

Let us recall the following bound from Lemma~\ref{lema} (it also holds for $\hat{\overline{a}}$, which is easy to check)
\begin{equation}
\label{eq:hata}
    0<\hat{a}_j,\ \hat{\overline{a}}_j < \frac{1}{\mu_k}.
\end{equation}

We will use the fact (Lemma~3.6) that $a_j$ and $\hat{a}_j$ is increasing as a function of $a_{j-1}$.

We use the following trivial lower bound for $a_j$'s
\begin{equation}
\label{eq:a1lower}
a_1 = \frac{d_N}{d_{N-1}d_N + \mu_k^2}\geq \frac{4\mu_k}{17\mu_k^2}\geq \frac{4}{17\mu_k}.
\end{equation}
We show that lower bound $\frac{4}{17\mu_k}$ holds also for $a_{j}$. We proceed by induction, assuming $a_{j-1}\geq\frac{4}{17\mu_k}$, we have
\[
a_{j} = \frac{ d_{N-2j+2} + a_{j-1}\mu_k^2}{ d_{N-2j+1}d_{N-2j+2} + a_{j-1}d_{N-2j+1}\mu_k^2 + \mu_k^2} \geq \frac{ d_{N-2j+2} + \frac{4\mu_k}{17}}{ d_{N-2j+1}d_{N-2j+2} + \sqrt{2}\mu_k  d_{N-2j+1} + \mu_k^2},
\]
denoting $\beta = d_{N-2j+1}$, we validate
\[
\frac{\beta+2 + \frac{4\mu_k}{17}}{\beta(\beta+2)+\sqrt{2}\mu_k\beta+\mu_k^2} \geq \frac{4}{17\mu_k},
\]
after simplifying we end up with
\[
(17-4\sqrt{2})\mu_k\beta + 26\mu_k-4\beta^2\geq 0,
\]
which is true for all $\beta<2\mu_k$ (we consider only rows $1:2\left\lfloor\mu_k\right\rfloor$).

Using the upper bound $a_{j-1} \leq \frac{\sqrt{2}}{\mu_k}$ from Lemma~\ref{lema}, \eqref{eq:a1lower}, and the fact that $a_j$ is increasing as a function of $a_{j-1}$ we obtain the following upper bound for $a_j$'s 
\[
a_j \leq \frac{d_{N-2j+2} + \sqrt{2}\mu_k}{d_{N-2j+2}d_{N-2j+1} + a_{j-1}d_{N-2j+1}\mu_k^2 + \mu_k^2} \leq \frac{d_{N-2j+2} + \sqrt{2}\mu_k}{d_{N-2j+2}d_{N-2j+1} + 4d_{N-2j+1}\mu_k/17 + \mu_k^2} = A_j.
\]

It also holds 

\begin{align}
    a_j&\leq \frac{d_{N-2j+2} + A_{j-1}\mu_k^2}{d_{N-2j+2}d_{N-2j+1} +4d_{N-2j+1}\mu_k/17 +\mu_k^2},\nonumber
\end{align}
where $A_{j-1} = \frac{d_{N-2j+4} + \sqrt{2}\mu_k}{d_{N-2j+3}d_{N-2j+4} + 4d_{N-2j+3}\mu_k/17+\mu_k^2}$.

By performing analogous computations it also holds that
\[
\hat{a}_{j}\leq \frac{\mu_k}{d_{N-2j+2}d_{N-2j+1}+4d_{N-2j+1}\mu_k/17+\mu_k^2}.
\]

We estimate the first sum in \eqref{eq:totalbd} by the definite integral, where we use substitution $y\bydef \frac{N}{2}-j$, we also use $\mu=\mu_k$ to simplify the notation. The range for $y$'s is fixed $y\in[1,\lfloor\mu_k\rfloor]$, as the norm in \eqref{eq:totalbd} concerns only rows $1:2\lfloor\mu_k\rfloor$, i.e. $j$ such that $N-2j+2\leq 2\lfloor\mu_k\rfloor$. We overestimate the finite sums by the definite integral over the wider and continuous range $y\in[0,\mu_k]$ as follows
\begin{multline*}
    \sum_{j=1}^{\left\lfloor\mu_k\right\rfloor}{a_j} \leq \int_{y=0}^{\mu}{\frac{2(2y+2)+A_{j-1}\mu^2}{4(2y+2)(2y+1)+8(2y+1)\mu/17+\mu^2}\,dy} \leq
    \int_{y=0}^{\mu}{\frac{4(y+1) + \left(\sqrt{2} \mu+4(y+2)\right)\mu^2}{\left(\mu^2+16 y^2\right)^2}\,dy} =\\
    \frac{\mu^3 \left(4 \left(4+\sqrt{2}\right)+17 \sqrt{2} \tan ^{-1}(4)\right)+8 \mu^2 \left(4+17 \tan ^{-1}(4)\right)+16 \mu+16+68 \tan ^{-1}(4)}{136 \mu^3}.
\end{multline*}
The maximal order w.r.t. $\mu$  is the same in the numerator as in denominator, hence the function above is decreasing w.r.t. $\mu$. The proof of the integral formula, a plot, and numerical evaluations for given $\mu$'s can be found in the attached Mathematica script.

It holds for example
\begin{subequations}
\label{eq:aksums}
\begin{align}
    \sum_{j=1}^{\left\lfloor\mu_k\right\rfloor}{a_j} &\leq 0.474\text{ for }\mu=10,\nonumber\\
    \sum_{j=1}^{\left\lfloor\mu_k\right\rfloor}{a_j} &\leq 0.404\text{ for }\mu=100,\label{bdfirst}\\
    \sum_{j=1}^{\left\lfloor\mu_k\right\rfloor}{a_j} &\leq 0.395\text{ for }\mu=1000,\nonumber
\end{align}
\end{subequations}

The second term in Lemma~\ref{lem:totalbd} is constant as
\[
   \sum_{j=1}^{\left\lfloor\mu_k\right\rfloor}{\hat{a}_j} \leq \int_{y=0}^{\mu}{\frac{\mu}{4(2y+2)(2y+1)+8(2y+1)\mu/17+\mu^2}\,dy} \leq \int_{y=0}^{\mu}{\frac{\mu}{16y^2+\mu^2}\,dy} = \frac{1}{4} \tan ^{-1}(4)
\]

Hence, 
\begin{equation}
\label{eq:akhatsum}
\sum_{j=1}^{\left\lfloor\mu_k\right\rfloor}{\hat{a}_j} \leq 0.332.
\end{equation}

Finally the explicit bounds given in Theorem~\ref{thm:tkimprov} follow from Lemma~\ref{lem:totalbd}, \eqref{eq:aksums}, and \eqref{eq:akhatsum}.
\qedsymbol